\documentclass[10pt, english]{amsart}
\newtheorem{theorem}{Theorem}[section]
\newtheorem{proposition}[theorem]{Proposition}
\newtheorem{definition}[theorem]{Definition}
\newtheorem{lemma}[theorem]{Lemma}

\newtheorem{example}[theorem]{Example}
\newtheorem{observation}[theorem]{Observatio
n}
\newtheorem{remark}[theorem]{Remark}
\newcounter{cases}
\newcounter{subcases}[cases]
\newenvironment{mycases}
{%
 \setcounter{cases}{0}%
 \setcounter{subcases}{0}%
 \def\case
 {%
  \par\noindent
  \refstepcounter{cases}%
  \textbf{Case \thecases.}
 }%
 \def\subcase
 {%
  \par\noindent
  \refstepcounter{subcases}%
  \textit{Subcase (\thesubcases):}
 }%
}
{%
 \par
}
\renewcommand*\thecases{\arabic{cases}}
\renewcommand*\thesubcases{\roman{subcases}}
\usepackage{enumitem}
\newlist{steps}{enumerate}{1}
\setlist[steps,1]{label=Step \arabic*:}
\usepackage[pagewise]{lineno}\linenumbers
\usepackage{amssymb,latexsym}
\usepackage{graphicx}
\usepackage{amsthm}
\numberwithin{equation}{section}
\usepackage{tikz}
\usepackage{tikz-cd}
\usepackage{float}
\usepackage{mathtools}
\usepackage[section]{placeins}
\usetikzlibrary{shapes.geometric, arrows}
\usepackage[utf8]{inputenc}
\usepackage{hyperref}
\hypersetup{
linktocpage=true,
}
\begin{document}
\nolinenumbers

\title{Additive actions on hyperquadrics of corank two}
\author{yingqi liu}
\date{\today}
\address{yingqi liu, Academy of Mathematics and Systems Science, Chinese Academy of Sciences, Beijing, 100190, China}
\email{liuyingqi@amss.ac.cn}
\subjclass[2010]{Primary 14L30; \ Secondary 14J50, 15A69}
\keywords{Commutative unipotent group action, hyperquadric of corank two, symmetric bilinear form.}              
\maketitle

\begin{abstract}
 For a projective variety $X$ in $\mathbb{P}^{m}$ of dimension $n$, an additive action on $X$ is an effective action of $\mathbb{G}_{a}^{n}$ on $\mathbb{P}^{m}$ such that $X$ is $\mathbb{G}_{a}^{n}$-invariant and the induced action on $X$ has an open orbit. Arzhantsev and Popovskiy have classified additive actions on hyperquadrics of corank 0 or 1. In this paper, we give the classification of additive actions on hyperquadrics of corank 2 whose singularities are not fixed by the $\mathbb{G}_{a}^{n}$-action.

\end{abstract}

\maketitle

\section{Introduction}
\subsection{Main results} Throughout the paper, we work over an algebraically closed field $\mathbb{K}$ of characteristic zero. Let $\mathbb{G}_{a}=(\mathbb{K},+)$ be the additive group of the field and $\mathbb{G}_{a}^{n}=\mathbb{G}_{a} \times \mathbb{G}_{a} \times ... \times \mathbb{G}_{a}$($n$ times) be the vector group. In this article we study additive actions on projective varieties defined as follows.
\begin{definition}
 Let X be a closed subvariety of dimension $n$ in  $\mathbb{P}^{m}$. An additive action on X is an effective algebraic group action $\mathbb{G}_{a}^{n}\times \mathbb{P}^{m} \rightarrow \mathbb{P}^{m}$ such that X is $\mathbb{G}_{a}^{n}$-invariant and the induced action $\mathbb{G}_{a}^{n}\times X \rightarrow X$ has an open orbit $O$. Two additive actions on X are said to be equivalent if one is obtained from another via an automorphism of $\mathbb{P}^{m}$ preserving $X$.
\end{definition}
In the following we represent an additive action on $X$ by a pair $(\mathbb{G}_{a}^{n},X)$ or a triple $(\mathbb{G}_{a}^{n},X,\textbf{L})$, where $\textbf{L}$ is the underlying projective space. We define $X \backslash O$ to be the boundary of the action and define $l(\mathbb{G}_{a}^{n},X)$ to be the maximal dimension of orbits in the boundary. For a group action of $G$ on a set $S$, we define the set of fixed points under the action to be $Fix(S)=\{x \in S \, | \, g \cdot x =x, \text{for any } g \in G \} $.  We say a subset in the projective space is non-degenerate if it is not contained in any hyperplane. \par
Recall that a $\mathbb{G}_{a}^{n}$-action on $\mathbb{P}^{m}$ is induced by a linear representation of $\mathbb{G}_{a}^{n}$, namely write $\mathbb{P}^{m}=\mathbb{P}V$ for an $(m+1)$-dimensional vector space $V$, then the action is given by:
 \begin{align*}
\mathbb{G}_{a}^{n} \times  \mathbb{P}V \mapsto  \mathbb{P}V \\
(g,[v])  \mapsto  [\rho(g)(v)]
\end{align*}
where $\rho : \mathbb{G}_{a}^{n} \mapsto GL(V)$ is a rational representation of the vector group $\mathbb{G}_{a}^{n}$. In \cite{HT} Hassett and Tschinkel showed that if the action is faithful and has a non-degenerate orbit $O$ in $\mathbb{P}^{m}$, then the vector space $V$ can be realized as a finite dimensional local algebra. They identified additive actions on projective spaces with certain  finite dimensional local algebras. The simplest additive action on a projective space is the one with fixed boundary, it is unique and can be given explicitly as follows.
 \begin{align*}
 \mathbb{G}_{a}^{m} \times \mathbb{P}^{m} \mapsto & \mathbb{P}^{m} \\
	(g_{1},...,g_{m}) \times [x_{0}:x_{1}:...:x_{m}] \mapsto & [x_{0}:x_{1}+g_{1}x_{0}:...:x_{m}+g_{m}x_{0}]
 \end{align*}
In \cite{AP} Arzhantsev and Popovskiy identified additive actions on hypersurfaces in $\mathbb{P}^{n+1}$ with invariant $d$-linear symmetric forms on $(n+2)$-dimensional local algebras. As an application they obtained classifications of additive actions on hyperquadrics of corank 0 and 1, where the corank of a hyperquadric $Q$ is the corank of the quadratic form defining $Q$. Given an additive action on a hyperquadric $Q$, if $corank(Q)=0$ (i.e., $Q$ is smooth), then the action is unique up to equivalences (also cf. \cite{Sharoyko2009HassettTschinkelCA}) and $l(\mathbb{G}_{a}^{n},Q)=1$. If $corank(Q)=1$, then the action is determined by a symmetric matrix up to an orthogonal transformation, adding a scalar matrix and a scalar multiplication (cf. \cite[Proof of Proposition 7]{AP}), namely for two  symmetric matrices $\Lambda$ and $\Lambda'$, they determine the same action if and only if there exist a nonzero $a \in \mathbb{K}$, $h \in \mathbb{K}$ and an orthogonal matrix $A$ (i.e., $A^\intercal A=I$) such that $\Lambda'=A^\intercal (a \Lambda+hI )A$. In this case, the action has fixed singular locus and $l(\mathbb{G}_{a}^{n},Q)=2$.\par 
In this paper, we study additive actions on hyperquadrics of corank 2. In this case the action is determined by two symmetric bilinear forms on a certain finite dimesnional local algebra. The singular locus, which is a projective line, is either fixed by the action or is the union of a orbit  and a fixed point.\par
When the singular locus is fixed by the $\mathbb{G}_{a}^{n}$-action,  it is a natural generalization of the case when $corank(Q)=1$. In this case, using a similar method as in \cite[Proposition 7]{AP} one can see that the action is determined by a pair of symmetric matrices up to a simultaneous orthogonal similarity and an affine transformation of pais of matrices, namely for two pairs of symmetric matrices $(\Lambda_{1},\Lambda_{2})$ and $(\Lambda'_{1},\Lambda'_{2})$, they determine the same action if and only if there exist $a_{11},a_{12},a_{21},a_{22},h_{1},h_{2} \in \mathbb{K}$ with $a_{11}a_{22}-a_{12}a_{21} \not=0$ and an orthogonal matrix $A$ such that:
\begin{align*}
\Lambda'_{1}=A^\intercal (a_{11} \Lambda_{1}+a_{12} \Lambda_{2}+h_{1}I )A\\
\Lambda'_{2}=A^\intercal (a_{21} \Lambda_{1}+a_{22} \Lambda_{2}+h_{2}I )A
\end{align*}
\begin{remark}
	For $\mathbb{K}=\mathbb{C}$, the problem of classifying pairs of matrices under simultaneous similarity is solved explicitly by Friedland \cite{Fr}. As an application, for almost all pairs of symmetric matrices $(A,B)$, the characteristic polynomial $|\lambda I-(A+xB)|$ determines a finite number of similtaneous orthogonal similarities classes. 
\end{remark} 
 In this paper we focus on the case when the action has unfixed singularities, our main observation is that under the identification, one of the bilinear forms vanishes on a certain hyperplane of the maximal ideal. As a result, the action can be recovered from two kinds of simpler actions which has been classified before. One is an action on a projective space with fixed boundary, the other one is an action on a hyperquadric of corank $r\geqslant 2$, which can be simply recovered from an action on a hyperquadric of corank one as follows.

\begin{definition}\label{simp_quad_action}
 Let $Q$ be a hyperquadric of corank one  in  $\textbf{P}=\mathbb{P}V$ with an additive action induced by $\rho: \mathbb{G}_{a}^{n} \mapsto GL(V)$. Choose an element $\alpha$ in the open orbit $O$. For any $r \geqslant 1$, viewing $\textbf{P}$ as a subspace of $\textbf{P}'=\mathbb{P}^{n+r}$ of codimension $r$ and write the coordinate of $\textbf{P}$ and $\textbf{P}'$ to be $[v]=[x_{0},x_{1}:...:x_{n}]$ and $[v,z]=[v:z_{1}:...:z_{r}]$ respectively, where $\alpha= [1:0:...:0]$. Let $\textbf{L}=\{v=0\} \subseteq  \textbf{P}'$ and  $\widetilde{Q}$ be the projective cone over $Q$ with vertex being $\textbf{L}$. Then we extend the action on $Q$ to $\widetilde{Q}$ as follows. \par
	Write $\mathbb{G}_{a}^{n+r}=\mathbb{G}_{a}^{n} \times \mathbb{G}_{a}^{r}=\{(g,h):g \in \mathbb{G}_{a}^{n}, h \in \mathbb{G}_{a}^{r}\}$  , then the action $(\mathbb{G}_{a}^{n+r},\widetilde{Q})$ is defined to be:
	\begin{align*}
	 \mathbb{G}_{a}^{n+r}  \times \widetilde{Q} \mapsto & \widetilde{Q} \\
	(g,h) \times [v:z] \mapsto & [\rho(g)(v):z+x_{0} \cdot h]
	\end{align*}
     If we extend the action using another element $\alpha' \in O$, then the induced action on $\widetilde{Q}$ is equivalent to the previous action through an linear isomorphism $\phi$ of $P'$ such that $\phi(P)=P,\phi(\alpha)=\alpha'$ and $\phi_{|_{L}}=id_{L}$. Hence the definition of the extended action on $\widetilde{Q}$ is unique up to equivalences. We call the extended action is simply recovered from the given action $(\mathbb{G}_{a}^{n},Q)$.
     
\end{definition}
\begin{remark}\label{rem_fix_sing_quad_dertermined}
Geometrically the action on $\widetilde{Q}$ is extended by the action on $Q$ through an action on a projective space with fixed boundary. Note that $\widetilde{Q}$ is contained in the linear span $<L_{\alpha},D>$, where  $L_{\alpha}$ is the cone over $\textbf{L}$ with the vertex being $\alpha$ ,  $D$ is the boundary $Q \backslash O$. Hence the action of $\mathbb{G}_{a}^{n+r}=\mathbb{G}_{a}^{n} \times \mathbb{G}_{a}^{r}$ on $\widetilde{Q}$ is  determined by its action on $L_{\alpha}$ and $D$, which is rather simple: the action of $\mathbb{G}_{a}^{n}$ on $L_{\alpha}$ and the action of $\mathbb{G}_{a}^{r}$ on $D$ are both trival while the action of $\mathbb{G}_{a}^{r}$ on $L_{\alpha}$ is an additive action on the projective space with fixed boundary.
\end{remark}
Now to recover a given action by a simpler action, we introduce an operation for any given additive action on a hyperquadric with unfixed singularities or an action on a projective space with unfixed boundary. We start with the following definition.
\begin{definition}
	  Let $X$ in $\mathbb{P}^{m}$ be a hyperquadric or a projective space with an additive action, $O$ being the open orbit.
	\[
	K(X)=
	\begin{cases*}
	Sing(X) &\text{if $X$ is a hyperquadric}\\
	X \backslash O & \text{if $X$ is a projective space}
	\end{cases*}
	\] 
\end{definition}
\begin{theorem}\label{degeneration_construction}
	For an additive action on $X$ in $\mathbb{P}^{m}$, where $X$ is either a hyperquadric or a projective space with open orbit $O$ such that $K(X) \nsubseteq Fix(X)$. Choose $x_{0} \in O$. Let $G^{(1)}=\cap_{x \in K(X)}G_{x}$ and let  $\textbf{L}^{(1)}$ be the linear span of $G^{(1)} \cdot x_{0}$, then:\\
	(i) $\textbf{L}^{(1)} \subsetneq \mathbb{P}^{m }$.\\
	(ii) $\textbf{L}^{(1)}$ is $G^{(1)}$-invariant and the action of $G^{(1)}$ on $\textbf{L}^{(1)}$ induces an additive action on $Q^{(1)}=\overline{G^{(1)}\cdot x_{0}}  \subseteq \textbf{L}^{(1)}$ with the open orbit $O^{(1)}=G^{(1)}\cdot x_{0}$, where $Q^{(1)}$ is either a non-degenerate hyperquadric or the whole projective space $\textbf{L}^{(1)}$.
\end{theorem}
We furtherly define when such an operation is effective for our classification.
\begin{definition}\label{effective}
	Let $Q$ be a hyperquadric with an additive action such that $Sing(X) \nsubseteq Fix(X)$, we say the operation obtained in Theorem \ref{degeneration_construction}: $(\mathbb{G}_{a}^{n},Q,\mathbb{P}^{n+1}) \mapsto (G^{(1)},Q^{(1)},\textbf{L}^{(1)})$ is effective if $K(Q) \subsetneqq K(Q^{(1)})$.
\end{definition}
Starting from a given additive action on the  hyperquadric $Q$ with unfixed singularities, the operation defined in Theorem 1.6 and the effective condition in Definition 1.7 give a procedure of reducing the present action to a lower dimensional one, which has to terminate as the dimension of the underlying projective space decreases strictly by Theorem 1.6 (i).  The procedure ends in three different ways,  which we call Type $A$, Type $B$ and Type $C$. We use the following flow chart to represent the procedure.  \par
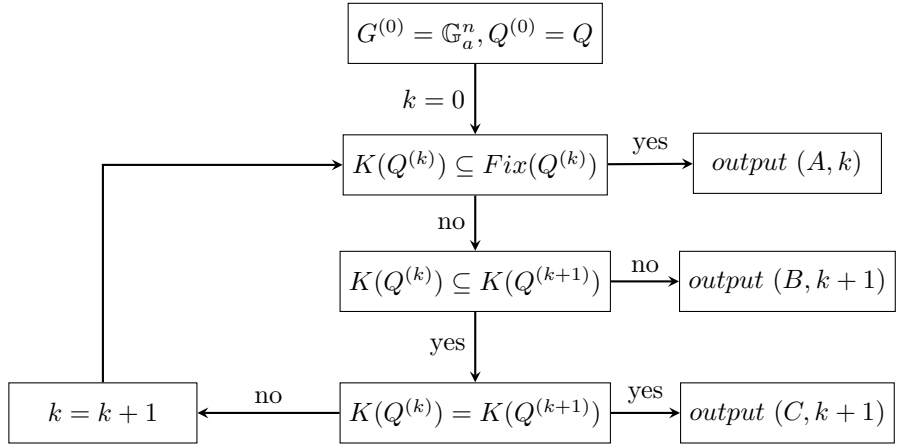
\begin{figure}[H]\label{diagram_geo}
	\centering
	\tikzstyle{io} = [trapezium, trapezium left angle=70, trapezium right angle=110, minimum width=2.5cm, minimum height=0.8cm, text centered, draw=black, fill=blue!30]
	\tikzstyle{process} = [rectangle, minimum width=2.5cm, minimum height=0.8cm, text centered, draw=black]
	\tikzstyle{decision} = [rectangle, minimum width=1.0cm, minimum height=0.8cm, text centered, draw=black]
	\tikzstyle{arrow} = [thick,->,>=stealth]
	\begin{tikzpicture}[node distance=1.15cm]   
	\node (pro0) [process] {$G^{(0)}=\mathbb{G}_{a}^{n},Q^{(0)}=Q$};
	\node (dec1) [decision,below of= pro0,yshift=-0.6cm] {$K(Q^{(k)}) \subseteq Fix(Q^{(k)})$}; 
	\node (pro1) [process,right of =dec1,xshift=3.0cm,yshift=0.6] {$output$ $(A,k)$};
	\node (dec2) [decision,below of = dec1,yshift=-0.4cm] {$K(Q^{(k)})\subseteq K(Q^{(k+1)})$};
	\node (pro2)[process,right of = dec2,xshift=3.0cm,yshift=0.0cm] {$output$ $(B,k+1)$};
	\node (dec3) [decision,below of =dec2,yshift=-0.6cm] {$K(Q^{(k)})=K(Q^{(k+1)})$};
	\node (pro4) [process,right of =dec3,xshift=3.0cm,yshift=0.0cm] {$output$ $(C,k+1)$};
	\node (pro5) [process,left of =dec3,xshift=-3.8cm,yshift=0.0cm] {$k=k+1$};
	
	\draw [arrow] (pro0) -- node[anchor=east]{$k=0$}(dec1);
	\draw [arrow] (dec1) -- node [anchor=south] {yes} (pro1);
	\draw [arrow] (dec1) -- node [anchor=east] {no} (dec2) ;
	\draw [arrow] (dec2) -- node [anchor=south]{no} (pro2);
	\draw [arrow] (dec2) -- node [anchor=east] {yes} (dec3);
	\draw [arrow] (dec3) -- node [anchor=south] {yes} (pro4);
	\draw [arrow] (dec3) -- node [anchor=south]  {no}  (pro5);
	\draw [arrow] (pro5) |- (dec1) ;
	\end{tikzpicture}
\end{figure}
\noindent where for each $k$, if $K(Q^{(k)}) \nsubseteq Fix(Q^{(k)})$, let $(G^{(k)},Q^{(k)})\mapsto (G^{(k+1)},Q^{(k+1)})$ be the operation obtained in Theorem \ref{degeneration_construction}.\par
We use $(x,t,G^{(t)},Q^{(t)})$ to represent the final output of the flow chart, where $(x,t)$ is the output of the flow chart and $(G^{(t)},Q^{(t)})$ is the corresponding action.\par 
In the case of corank two, the following theorem shows that the flow chart conversely gives the explicit process of recovering and together with $l(G_{a}^{n},Q)$ the final output determines the action up to equivalences. 
\begin{theorem}\label{determine_by_degeneration}
	Let $Q$ be an hyperquadric of corank two with an additive action, assume the action has  unfixed singularities and $dim(Q) \geqslant 5$, let $(x,t,G^{(t)},Q^{(t)})$ be the final output of the flow chart above. Then:\par
	(i) $(G^{(t)},Q^{(t)})$ is either an action on a projective space with fixed boundary or an action on  a hyperquadric given in Definition \ref{simp_quad_action}.\par
	(ii)  $l(\mathbb{G}_{a}^{n},X) \leqslant 3$ and $codim(Q^{(k+1)},Q^{(k)})=1$, for any $k \leqslant t-1$.\par
	(iii) if $(\mathbb{G}_{a}^{n},\widetilde{Q})$ is another additive action on the hyperquadric of corank two $\widetilde{Q}$ with unfixed singularities and $dim(\widetilde{Q}) \geqslant 5$, let $(x',t',\widetilde{G}^{(t')},\widetilde{Q}^{(t')})$ be the final output of the flow chart, then $(\mathbb{G}_{a}^{n},Q)$ is equivalent to $(\mathbb{G}_{a}^{n},\widetilde{Q})$ if and only if $l(\mathbb{G}_{a}^{n},Q)=l(\mathbb{G}_{a}^{n},Q'),x=x',t=t'$ and $(G^{(t)},Q^{(t)})$ is equivalent to $(\widetilde{G}^{(t')},\widetilde{Q}^{(t')})$.
\end{theorem}
	 Combining Remark \ref{rem_fix_sing_quad_dertermined} with classification of actions on hyperquadrics of corank one, we can determine the output action $(G^{(t)},Q^{(t)})$ explicitly. Then by Theorem \ref{determine_by_degeneration}, we can give classification of additive actions on hyperquadrics of corank two with unfixed singularities in terms of the final output of the flow chart.  
\begin{theorem}\label{classfication}
	Let $Q$ be a hyperquadric of corank two, then additive action on $Q$ with unfixed singularities has equivalence type as follows:\par
	(a) $dim(Q) \geqslant 5$. Let the final output in the flow chart be $(x,t,G^{(t)},Q^{(t)})$ 
	then we separate it into 8 different types with respect to the value of $x,t$ and whether $Q^{(t)}$ is a projective space or a hyperquadric:\par
	(a.1) Type $x_{0}$: if $x \in \{B,C\}$   and $t=1$.\par
	(a.2) Type $x_{1}$: if $x \in \{A,B,C\}$, $t \geqslant 2$ when $x \in \{B,C\}$ and  $Q^{(t)}$ is a projective space.\par
	(a.3) Type $x_{2}$: if $x \in \{A,B,C\}$, $t \geqslant 2$ when $x \in \{B,C\}$ and  $Q^{(t)}$ is a hyperquadric.\par	
    (b) $dim(Q) \leqslant 4$: there are 14 different types.
\end{theorem}
\begin{remark}
Explicit classification result of each type will be given in Proposition \ref{classify_A},\ref{classify_B},\ref{classify_C} and Section 4.2 in terms of the algebraic structure of finite dimensional local algebras.
\end{remark}
The simplest types are Type $B_{0}$ and Type $C_{0}$, i.e., Type $x_{0}$ for $x \in \{B,C\}$. They can be directly recovered from an additive action on a hyperquadric of corank one. Here we describe actions of Type $B_{0}$ as an example.
\begin{example}
  Let $Q$  be a hyperquadric of corank two in $\mathbb{P}^{n+1}=\mathbb{P}V$ with an additive action, assume $dim(Q) \geqslant 5$ and $Sing(Q) \nsubseteq Fix(Q)$, consider $(G^{(1)},Q^{(1)},\textbf{L}^{(1)})$  obtained in Theorem \ref{degeneration_construction}. If it is of Type $B_{0}$ then: \par
	(i)  $\, Q^{(1)}$ is a hyperquadric of corank one in $\textbf{L}^{(1)}$. \par
	(ii) choose any $\alpha $ in the open orbit $O$ and any $\alpha' \in Sing(Q) \backslash Fix(Q)$ there exist suitable coordinate $\{x_{0},x_{1},x_{2},...,x_{n-1},y_{0},y_{1}\}$ of $\mathbb{P}^{n+1}$ w.r.t the basis of $V$, $\alpha_{0},\alpha_{1},..,\alpha_{n-1},\beta_{0},\beta_{1}$, such that $\alpha=[\beta_{1}]$, $\alpha'=[\alpha_{1}]$, $\textbf{L}^{(1)}=\{x_{1}=0\}$, $Q^{(1)}=\textbf{L}^{(1)} \cap Q$ and 
	\begin{align*}
	Q=\{x_{2}^{2}+...+x_{n-1}^{2}+y_{0} \cdot y_{1}=0\}.
	\end{align*}
	Moreover for $V'=\langle \alpha_{0},\alpha_{2},..,\alpha_{n-1},\beta_{0},\beta_{1} \rangle$ such that $\textbf{L}^{(1)}=\mathbb{P}V'$, let the action $(G^{(1)},Q^{(1)})$ be given by:
	\begin{align*}
	G^{(1)} \times  \textbf{L}^{(1)} \mapsto \textbf{L}^{(1)} \\
	(a,[v'])  \mapsto  [\rho(a)(v')]
	\end{align*}
	 where $\rho : G^{(1)} \rightarrow GL(V')$ is a rational representation of $G^{(1)}$ .\par
	 Then there is a decomposition of $\mathbb{G}_{a}^{n}=G^{(1)} \oplus \mathbb{G}_{a}$ such that if we write\\ $a=(a_{0},a_{2},..,a_{n-1}) \in G^{(1)}$, $s \in G_{a}$, $v=(x_{0},x_{1},..,x_{n-1},y_{0},y_{1}) \in V$ and $v'=(x_{0},0,..,x_{n-1},y_{0},y_{1}) \in V'$ then the action $(\mathbb{G}_{a}^{n},Q)$ is given by:
	\begin{align*}
	\mathbb{G}_{a}^{n} \times  \mathbb{P}^{n+1} \mapsto & \mathbb{P}^{n+1   } \\
	((a,s) \times [v])  \mapsto & [v'']
	\end{align*}
	where $v''=\rho(a)(v')+(\frac{s^{2}y_{1}}{2}+sx_{1}) \cdot \alpha_{0}+(sy_{1}+x_{1}) \cdot \alpha_{1}$.
	
\end{example}
\subsection{Notation and conventions}
Throughout the article, in a given finite dimensional local algebra $R$, we use $\alpha \cdot \beta$ to represent multiplication between two elements in $R$, where $\alpha$ can also be taken as a scalar in $\mathbb{K}$. Furthermore we define the following:\par
(a) if $\alpha \in R$, $V\subseteq R$, then $	\alpha \cdot V \doteq \{\alpha \cdot \beta: \beta \in V\}$\par
(b) if $V,V' \subseteq R$, then $V \cdot V' \doteq \{\sum_{i=1}^{n} \alpha_{i} \cdot \alpha'_{i}: n \in \mathbb{N}, \alpha_{i} \in V, \alpha'_{i} \in V'\}$.
\subsection{Outline of the classification} Given an additive action on a hyperquadric $Q$ in $\mathbb{P}^{n+1}$, there is an $(n+2)$-dimensional local algebra $R$ with a hyperplane $W$ of the maximal ideal $\mathfrak{m}$ and a bilinear form $F$ on R such that 
\begin{align*}
\mathbb{P}^{n+1}=\mathbb{P}(R),Q=\mathbb{P}(\{r \in R: F(r,r)=0\}),
\end{align*} 
and if we choose a basis of $W$, $w_{1},...,w_{n}$, then the action is given by (up to equivalences): 
\begin{align*}
\mathbb{G}_{a}^{n} \times R \mapsto & \,R \\
((a_{1},a_{2},...,a_{n}),  r) \mapsto & \, r \cdot exp(a_{1}w_{1}+...+a_{n}w_{n}).
\end{align*}
 Hence to classify additive actions is equivalent to classify algebraic structures of the triple $(R,W,F)$.  Note that $Sing(Q)=\mathbb{P}(Ker(F))$, furthermore we show that $Ker(F) \subseteq W$ and if we choose a basis of $Ker(F)$, $\mu_{1},..,\mu_{l}$, and choose any $b_{0} \in \mathfrak{m}^{2} \backslash W$ then we can represent the multiplications of elements in $\mathfrak{m}$ as follows:
 \begin{equation}
 a \cdot a'=F(a,a') b_{0} + V_{1}(a,a') \mu_{1} +V_{2}(a,a') \mu_{2}+...+V_{l}(a,a') \mu_{l},
 \end{equation}
 for any $a,a' \in \mathfrak{m}$, where $\{V_{i}: 1 \leqslant  i \leqslant l\}$ is a set of symmetric bilinear forms on $R$. When the corank equals one we have $l=1$ and $\mu_{1} \cdot \mathfrak{m}=0$, also one can choose $b_{0}$ s.t. $b_{0} \cdot \mathfrak{m}=0$. Hence if we extend $\mu_{1}$ to a basis of $W$ namely $\mu_{1},e_{1},..,e_{n-1},$ s.t. $F(e_{i},e_{j})=\delta_{i,j}$ then the multiplication in $\mathfrak{m}$ depends on the matrix $\Lambda=(V_{1}(e_{i},e_{j}))$. Also note that an orthogonal tranformation of the basis $e_{i}$'s with respect to $F$ leads to an orthogonal transformation of the matrix $\Lambda$, hence the classification of the action of corank one is reduced to normalize a symmetric bilinear forms under orthogonal transforamtions (cf. \cite[Proposition 7]{AP} and \cite[Chapter XI, \S 3]{matrices}).
 \par
 When the corank equals two, we can still choose $b_{0}$ s.t. $b_{0} \cdot \mathfrak{m}=0$. We note that in this case the condition $Sing(Q) \nsubseteq Fix(Q)$ enables us to use the idea of the case of corank one.\par
 Firstly we show that $Sing(Q) \nsubseteq Fix(Q)$ is equivalent to $Ker(F) \cdot W \not =0$. And if $Ker(F) \cdot W \not=0$, then we can furtherly define a hyperplane $V^{(1)}$ in $W$: 
 \begin{align*}
 V^{(1)}&=\{\alpha \in W: \alpha  \cdot Ker(F)=0\}\\
 V_{(1)} &=Ker(F_{|_{V^{(1)}}})
 \end{align*}
 By using the correspondence between additive actions and finite dimensional local algebras we show that if $V_{(1)}=V^{(1)}$ then the action $(G^{(1)},Q^{(1)},\textbf{L}^{(1)})$ obtained in Theorem \ref{degeneration_construction} is an action on a projective space and it corresponds to  $(R^{(1)},V^{(1)})$, where $R^{(1)}=V^{(1)} \oplus \langle 1_{R} \rangle$. If $V^{(1)} \not =V_{(1)}$ we show that the action $(G^{(1)},Q^{(1)},\textbf{L}^{(1)})$ corresponds to the triple $(R^{(1)},V^{(1)},F^{(1)})$ where $R^{(1)}=V^{(1)} \oplus \langle b_{0},1_{R} \rangle$, $F^{(1)}=F_{|_{V^{(1)}}}$ and $Q^{(1)}$ is a hyperquadric.\par
  Then our first key step is to show that after choosing suitable $\mu_{1} \in Ker(F)$, we have
  \begin{align*}
  V^{(1)} \cdot W \subseteq \langle \mu_{1},b_{0} \rangle,
  \end{align*} 
 which shows that the bilinear form $V_{2}$ defined in (1.1) vanishes on $V^{(1)}$. For the obtained subspaces $V_{(1)} \subseteq V^{(1)} \subseteq W$, our second key step is that
   if $Ker(F) \nsubseteq V^{(1)}$ (resp. $V_{(1)}=Ker(F)$), which geometrically means $K(Q) \nsubseteq K(Q^{(1)})$ (resp. $K(Q)= K(Q^{(1)})$), then we can directly normalize the multiplications in $\mathfrak{m}$. As a result, we recover action $(\mathbb{G}_{a}^{n},Q)$ from the action $(G^{(1)},Q^{(1)})$, which is an action given in Definition \ref{simp_quad_action}. This corresponds to an output of Type $B_{0}$ or $C_{0}$ in the flow chart.\par
  Otherwise we show that $codim(Ker(F),V_{(1)})=1$ and we furtherly consider the new action $(G^{(1)},Q^{(1)},L^{(1)})$, for which we separate into two more subcases.\par 
  (1) If $V_{(1)} \cdot V^{(1)}=0$ then we are in a situation similar to the case of corank one: $V_{2}=...=V_{l}=0$, $V_{(1)} \cdot V^{(1)}=0$, $b_{0} \cdot V^{(1)}=0$. Hence we can normalize the multiplications in $\mathfrak{m}^{(1)}=V^{(1)} \oplus \langle b_{0} \rangle$. As a result, we recover action $(\mathbb{G}_{a}^{n},Q)$ from the action $(G^{(1)},Q^{(1)})$, which is an action given in Definition \ref{simp_quad_action}. This corresponds to an output $(A,1)$ in the flow chart.\par
  (2) If $V_{(1)} \cdot V^{(1)} \not=0$ then we are in a situation similar to $Ker(F) \cdot W \not=0$, except that in this case the action is on a hyperquadric of corank three. On the other hand, we have $Ker(F) \cdot V^{(1)}=b_{0} \cdot V^{(1)}=0$ and $V_{2}=V_{3}=0$, hence the uncertainity of multiplications in $\mathfrak{m}^{(1)}$ is still one dimensional. For this reason we furtherly define
   \begin{align*}
  V^{(2)}& =\{\alpha \in V^{(1)}: \alpha  \cdot V_{(1)}=0\}\\
  V_{(2)} &=Ker(F_{|_{V^{(1)}}})
  \end{align*} 
  This corresponds to a new action $(G^{(2)},Q^{(2)},L^{(2)})$ in the flow chart, with $L^{(2)} \subsetneqq L^{(1)}$. Similarly we show that if $V_{(1)} \nsubseteq V_{(2)}$ or $V_{(1)} = V_{(2)}$ or $V_{(1)} \subsetneq V_{(2)}$ with $V_{(2)} \cdot V^{(2)}=0$ , then we can already normalize the multiplications in $\mathfrak{m}^{(1)}$. Otherwise we find $V_{(2)} \cdot V^{(2)}\not =0$ then as before we can furtherly define $(V^{(3)},V_{(3)})$ with $V^{(3)} \subsetneqq V^{(2)}$, and check whether it satisfies the conditions to be normalized. The discussion will be continued as above until we find that the present action satisfies the condition to be normalized i.e., to obtain an output in the flow chart, the procedure has to terminate as the dimension of $V^{(i)}$ decreases strictly. As a result we show that the output action is either an action on a projective space with fixed boundary or an action given in Definition \ref{simp_quad_action}. Moreover we obtain a chain of subspaces in $W$ corresponding to the flow chart:  
   \begin{align*}
   Ker(F) \subseteq V_{(1)} \subseteq ... V_{(s)} \subseteq V^{(s)} \subseteq ... \subseteq V^{(1)} \subseteq W,
   \end{align*}
   where $s=t$ if $x=A$, $s=t-1$ if $x=B$ or $x=C$.\par
   Then it remains to normalize the multiplications between elements outside $V^{(s)}$. This is completed through more technical operations shown in Lemmas \ref{norm_A}, \ref{norm_B} and \ref{norm_C}. After the normalization of the structure of $R$, we show the uniqueness of the normalized structure up to equivalences, which proves Theorem \ref{determine_by_degeneration} (iii). And the normalized structure of $(R,W,F)$ gives the explicit result of our classification of actions when $dim(Q) \geqslant 5$. Finally when $dim(Q) \leqslant 4$ we give the classification case by case.\par 
The article is organized as follows: in Section 2, we recall the correspondence between additive actions and finite dimensional local algebras; in Section 3 we first prove Theorem \ref{degeneration_construction} to obtain the action $(G^{(1)},Q^{(1)})$. Then we show that the existence of unfixed singularities will lead to $V^{(1)} \cdot W \subseteq  \langle  \mu_{1},b_{0} \rangle $ and we normalize the algebraic structure of $(R,W,F)$ when the type is $B_{0}$ or $C_{0}$; in Section~4, we first normalize the structure of $R$, then we show the uniqueness of the normalized structure, which gives proof of Theorem \ref{determine_by_degeneration} and also gives explicit result of our classification result shown in Theorem~\ref{classfication}.

\section{Additive actions and finite dimensional local algebras}
 As mentioned before an additive action $(\mathbb{G}_{a}^{n},X,\mathbb{P}^{m})$ is induced by a faithful rational linear representation $\rho: \mathbb{G}_{a}^{n} \rightarrow GL_{m+1}(\mathbb{K})$. Furtherly if $X$ is non-degenerate in $\mathbb{P}^{m}$ then $\rho$ becomes a cyclic representation i.e., $\langle \rho(g) \cdot v :g \in \mathbb{G}_{a}^{n}\rangle=\mathbb{K}^{m+1}$ for some nonzero $v \in \mathbb{K}^{m+1}$. Hassett and Tschinkel in \cite{HT} gave a complete characterization of such representations.
\begin{theorem}[\cite{HT},Theorem 2.14]\label{H-T}
There is 1-1 correspondence between the following two classes: \\
(1) equivalence classes of faithful rational cyclic representation $\rho: \mathbb{G}_{a}^{n} \mapsto GL_{m+1}(\mathbb{K})$;
(2) isomorphism classes of $(R,W)$, where $R$ is a local (m+1)-dimensional algebra with maximal ideal $\mathfrak{m}$ and $W$ is an n-dimensional subspace of $\mathfrak{m}$ that generates $R$ as an alegbra with unit.
\end{theorem}
\begin{remark}\label{action_algebrai}
	Under this correspondence a representation of $\mathbb{G}_{a}^{n}$ on $\mathbb{K}^{m+1}$ can always be viewed as an action on a local algebra $R \cong \mathbb{K}^{m+1}$. Moreover if we choose a $\mathbb{K}$-basis of $W$: $W=\langle w_{1},...,w_{n} \rangle$ then we can write down the action explicitly:
	\begin{align*}
	\mathbb{G}_{a}^{n} \times R \mapsto & \,R \\
	(g_{1},g_{2},...,g_{n}) \times r \mapsto & \, r \cdot exp(g_{1}w_{1}+...+g_{n}w_{n}).
	\end{align*}
And the induced action of the Lie algebra $\mathfrak{g}(G_{a}^{n})=G_{a}^{n}$ on $R$ is:
\begin{align*}
\mathfrak{g} \times R \mapsto & \,R \\
(g_{1},g_{2},...,g_{n}) \times r \mapsto & \, r \cdot (g_{1}w_{1}+...+g_{n}w_{n}),
\end{align*}
 we identify $\mathfrak{g} \cong W$ as vector spaces. 
\end{remark}
 Moreover Hassett and Tschinkel proved in \cite{HT} and later Arzhantsev and Popovskiy proved in \cite{AP} the following 1-1 correspondences.
\begin{theorem}[\cite{HT}, Proposition 2.15]\label{correspondence _proj} There's a bijection between the following two classes:\\
(1) equivalence classes of additive actions on $\mathbb{P}^{n}$;\\
(2) equivalence classes of (n+1)-dimensional local commutative algebras.
\end{theorem}
Under the correspondence the action is given as in Remark 2.2, where the subspace $W$ is the maximal ideal of the local algebra. 
\begin{theorem}[\cite{AP}, Proposition 3]\label{correspondence_hypersurface}  There's a bijection between the following two classes:\\
(1) equivalence classes of additive actions on hypersurfaces $H$ in $ \mathbb{P}^{n+1}$ of degree at least two;\\
(2) equivalence classes of $(R,W)$, where $R$ is a local (n+2)-dimensional algebra with maximal ideal $\mathfrak{m}$ and $W$ is a hyperplane of $\mathfrak{m}$ that generates the algebra $R$ with unit.
\end{theorem}
Then in \cite{AP} they furtherly introduced the notion of $invariant$ $multilinear$ $form$ for a pair $(R,W)$.
\begin{definition}[\cite{AP}, Definition 3]
	Let R be a local algebra with maximal ideal $\mathfrak{m}$. An invariant d-linear form on $R$ is a d-linear symmetric map
	\begin{equation*}
	F:R \times R \times ... \times R \mapsto \mathbb{K}
	\end{equation*}
	such that $F(1,1,...,1)=0$, the restriction of $F$ to $\mathfrak{m} \times ... \times \mathfrak{m}$ is nonzero, and there exist a hyperplane $W$ in $\mathfrak{m}$ which generates the algebra $R$ and such that:
	\begin{equation*}
	F(ab_{1},b_{2},...,b_{d})+F(b_{1},ab_{2},...,b_{d})+...+F(b_{1},b_{2},...,ab_{d})=0 \,\,\, \forall a \in W, b_{1},...,b_{d} \in R.
	\end{equation*}\par
	We say $F$ is irreducible if it can not be represented as product of two lower dimensional forms.
\end{definition}
Now given an additive action on a hypersurface $H=\{ f(x_{0},..,x_{n+1})=0\} \subseteq \mathbb{P}^{n+1}$, then under the correspondence in Theorem \ref{correspondence_hypersurface} the polarization $F$ of $f$ is an invariant multilinear form on $(R,W)$, which induces the following more explicit correspondence.
\begin{theorem}[\cite{AP}, Theorem 2]\label{correpondence} There is a bijection between the following two classes:\\
(1) equivalence classes of additive actions on hypersurface $H \subseteq \mathbb{P}^{n+1}$ of degree at least two;\\
(2) equivalence classes of (R,F), where R is a local algebra of dimension $n+2$ and F is an irreducible invariant d-linear form on $R$ up to a scalar.
\end{theorem}
Under the correspondence $\mathbb{P}^{n+1}=\mathbb{P}(R)$, $H=\mathbb{P}(\{r \in R: F(r,r,...,r)=0\})$, and the action on $\mathbb{P}^{n+1}$ corresponds to the action on $R$ as shown in Remark \ref{action_algebrai}, with the open orbit $O=\mathbb{P}(G_{a}^{n} \cdot 1_{R})$. Moreover $F$ is determined by $(R,W)$ as follows.
\begin{lemma}[\cite{AP}, Lemma 1]\label{F_det_by_R_W}
	Fix a $\mathbb{K}$-linear automorphism $\mathfrak{m}/W \cong \mathbb{K} $ with the projection $\pi: \mathfrak{m} \rightarrow  \mathfrak{m}/W \cong \mathbb{K}$ then the corresponding invariant linear form is (up to a scalar):
	\begin{equation*}	
	F_{W}(b_{1},...,b_{d})=(-1)^{k}k!(d-k-1)!\pi(b_{1}...b_{d}),
	\end{equation*}
	where k is the number of units among $b_{1},...,b_{d}$ and for $k=d$ let $F_{W}(1,1,...,1)=0$.
\end{lemma}
In the following we focus on additive actions on hyperquadrics, i.e., $d=2$ and we use a triple $(R,W,F)$ to represent an additive action on a hyperquadric $Q$ where $F$ is the bilinear form given in Theorem \ref{correpondence}. By  Lemma \ref{F_det_by_R_W} we have the following.
\begin{lemma}\label{y0}
Fix $b_{0} \in \mathfrak{m} \backslash W$ and the projection $y_{0}: R \rightarrow \mathbb{K}$ s.t. $y_{0}(1_{R})=y_{0}(W)=0$ and $y_{0}(b_{0})=1$. Then  for $ a,a' \in \mathfrak{m}$ and for $r \in W$ we have:
\begin{align*}
&F(a,a') =y_{0}(aa').\\
&F(1,1)=F(1,r)=0,\,F(1,b_{0})=-1.
\end{align*}
\end{lemma}
As $F$ is the polarization of the homogenous polynomial defining $Q$ we have $Sing(Q)=\mathbb{P}(Ker(F))$. Moreover we have the following.
\begin{lemma}\label{ker_kerr}
	$Ker(F) \subseteq W$ and $Ker(F|_{W})=Ker(F)$.
\end{lemma}
\begin{proof}
	By \cite[Theorem 5.1]{Arzhantsev2009HassettTschinkelCM}, the degree of the hypersurface is the maximal exponent $d$ such that $\mathfrak{m}^{d} \nsubseteq W$, for $d=2$ we have $\mathfrak{m}^{2} \nsubseteq W $ and $\mathfrak{m}^{3} \subseteq W$. Hence we can take a $b_{0} \in \mathfrak{m}^{2} \backslash W$ and the projection $y_{0}$ defined in Lemma \ref{y0}. \par
	For any $l \in Ker(F)$, write $l=a+tb_{0}+l_{W}$ for some $a,t \in \mathbb{K}$ and $l_{W} \in W$, then $t=-F(1,l)=0$ by Lemma \ref{y0}. And
	\begin{equation*}
    0=F(b_{0},l)=-a+F(b_{0},l_{W})=-a+y_{0}(b_{0}l_{W})=-a
    \end{equation*}
    as $b_{0}l_{W} \in \mathfrak{m}^{3} \subseteq W$, concluding that $l=l_{W} \in W$.\par
    For any $l \in Ker(F|_{W})$, then $F(1,l)=0$ as $l \in W$ and $F(l,b_{0})=y_{0}(lb_{0})=0$ as $b_{0}l \in \mathfrak{m}^{3} \subseteq W$ and $y_{0}(W)=0$, concluding that $l \in Ker(F)$.
\end{proof}
\begin{lemma}\label{basic_multi}
 For any  $b_{0} \in \mathfrak{m}^2 \backslash W$, $\mathfrak{m}^{2} \subseteq Ker(F)\oplus \langle b_{0} \rangle$.
\end{lemma}
\begin{proof}
	Firstly choose a $b_{0}  \in \mathfrak{m}^{2} \backslash W$. Given any $a,a' \in \mathfrak{m}$ we have
	\begin{equation*}
	aa'= y_{0}(aa')\cdot b_{0}+(aa')_{W},
	\end{equation*}
	  where $y_{0}$ is the projection defined in Lemma \ref{y0}. Now for any $r \in \mathfrak{m}$ then
	  \begin{equation*}
	  r \cdot (aa')_{W} =r \cdot (aa'- y_{0}(aa')\cdot b_{0}) \in \mathfrak{m}^3 \subseteq W,
	  \end{equation*}
	  as $b_{0} \in \mathfrak{m}^{2}$. Hence by Lemma \ref{y0}
	   \begin{equation*}
	    F((aa')_{W},r)=y_{0}((aa')_{W}\cdot r)=0,
	   \end{equation*}
	   as  $r  \cdot (aa')_{W} \in W$. Note that $F(1,(aa')_{W})=0$ since $F(1,W)=0$. It follows that $(aa')_{W} \in Ker(F)$, concluding the proof.\\
\end{proof}
	From above lemmas we can thus choose a $b_{0} \in \mathfrak{m}^{2} \backslash W$ such that $F(1,b_{0})=-1$ and $\mathfrak{m}^{2} \subseteq Ker(F)\oplus \langle b_{0} \rangle$. Moreover if we fix a basis of $Ker(F)=\langle \mu_{1},...,\mu_{l} \rangle$ then we can represent the multiplications of elements in $\mathfrak{m}$ as follows.
	\begin{equation}\label{multi}
	aa'=F(a,a') b_{0} + V_{1}(a,a') \mu_{1} +V_{2}(a,a') \mu_{2}+...+V_{l}(a,a') \mu_{l}.
	\end{equation}
\section{Unfixed singularities and vanishing of bilinear forms}
In this section, we first prove Theorem \ref{degeneration_construction}. Then we show that in the case of corank two the existence of unfixed singularities leads to $V^{(1)} \cdot W \subseteq \langle b_{0},\mu_{1} \rangle$. Finally we show that if $K(Q) \nsubseteq K(Q^{(1)})$ or $K(Q)=K(Q^{(1)})$ then we can already normalize the algebraic structure of $(R,W,F)$.\par
\subsection{Operation for actions with unfixed singularities}
 We first give an algebraic characterization of related concepts. Given an additive action on hyperquadric $Q$ represented by $(R,W,F)$, recall that $Sing(Q)=\mathbb{P}(Ker(F))$, $G^{(1)}=\cap_{x \in K(X)} G_{x }$ , $V^{(1)}=\{r' \in W \, | \, r' \cdot Ker(F)=0 \}$ and $V_{(1)}=Ker(F_{|_{V^{(1)}}})$. We furtherly  define $S'=\{r \in R \,|\, r \cdot W=0 \}$. Then we have the following.
\begin{proposition}\label{identify_2}
	(i) $\,Fix(Q)=\mathbb{P}(S')$.\par
	(ii) $\,G^{(1)}=exp(\mathfrak{g}^{(1)})$, where $\mathfrak{g}^{(1)} \subseteq \mathfrak{g}(\mathbb{G}_{a}^{n})$ is a Lie subalgebra and $\mathfrak{g}^{(1)} \cong V^{(1)}$ under the identification $\mathfrak{g}(\mathbb{G}_{a}^{n}) \cong W $ given in Remark \ref{action_algebrai}.\par
	(iii) $Ker(F) \cdot \mathfrak{m} \not=0 $ if and only if  $ V^{(1)} \not = W $ if and only if $ Sing(Q) \nsubseteq Fix(Q)$.
\end{proposition}
\begin{proof}
(i)	
By Remark \ref{action_algebrai}, the action of $\mathfrak{g}=\mathfrak{g}(\mathbb{G}_{a}^{n})$ on $R$ is given by multiplying elements of $W$ to $R$. Hence we have:
	\begin{align*}
		S'=\{r \in R: r \cdot W =0\}=\{r \in R: \mathfrak{g} \cdot r =0\}
	\end{align*}
	Also by Remark \ref{action_algebrai}, the action of $\mathbb{G}_{a}^{n}$ on $\mathbb{P}^{n+1}$ is identified with the action on $R$. Hence we have:
	\begin{equation*}
	Fix(Q)=\mathbb{P}(\{r \in R: g \cdot r =r, \, \forall g \in G_{a}^{n}\})=\mathbb{P}(\{r \in R: x \cdot r =0, \, \, \forall x \in \mathfrak{g}\})=\mathbb{P}(S').
	\end{equation*}
(ii) Similarly for the isotropy group $G^{(1)}$ of $Sing(Q)$ we have:
\begin{align*}
G^{(1)}&=\{g \in \mathbb{G}_{a}^{n}: g \cdot x=x, \, \forall x \in Sing(Q)\}=\{g \in \mathbb{G}_{a}^{n}: g \cdot r= r, \, \forall r \in Ker(F)\}\\
&=exp(\{x \in \mathfrak{g}: x \cdot r =0 , \, \forall r \in Ker(F)\}).
\end{align*}
Then by Remark \ref{action_algebrai}, under the identification of $\mathfrak{g} \cong W$, we have $\{x \in \mathfrak{g}: x \cdot r =0 , \, \forall r \in Ker(F)\} \cong \{r' \in W: r'\cdot r =0 , \, \forall r \in Ker(F)\}=V^{(1)}$.\par
(iii) The first equivalence follows from the definition of $V^{(1)}$ and the fact that $\mathfrak{m}$ can be generated by $W$. For the second equivalence, we have
$Sing(Q) \subseteq Fix(Q)$ if and only if $\mathbb{G}_{a}^{n}=G^{(1)}$ if and only if $\mathfrak{g}=\mathfrak{g}^{(1)}$ if and only if $V^{(1)}=W$,
  where the last equivalence follows from (ii). \par   
\end{proof}
Next we introduce a lemma to describe multiplications between elements in $\mathfrak{m}$ and elements in $Ker(F)$.
\begin{lemma}\label{basis0}
(i) $Ker(F) \cdot \mathfrak{m} \subseteq Ker(F)$ and there exist a $\mathbb{K}$-basis of $Ker(F)$, $\mu_{1},\mu_{2},...,\mu_{l}$, such that $\mu_{i} \cdot \mathfrak{m} \subseteq \langle \mu_{1},...,\mu_{i-1} \rangle $.  (ii) $V^{(1)} \not =0$. 
\end{lemma}
\begin{proof}
(i)	First note that $Sing(Q)$ is $\mathbb{G}_{a}^{n}$-stable. Then by Theorem \ref{correpondence} and  $\mathbb{P}(Ker(F))=Sing(Q)$, $Ker(F)$ is a $\mathbb{G}_{a}^{n}$-invariant subspace, hence by Remark \ref{action_algebrai} and the fact that $\mathfrak{m}$ is generated by $W$ we conclude that $Ker(F) \cdot \mathfrak{m} \subseteq Ker(F)$ \par  
	Now we choose a $\mathbb{K}$-basis of $\mathfrak{m}$ to be $S_{0}$, then for any $ c \in S_{0}$ we can define a linear map induced by multiplications:
	\begin{align*}
	\phi_{c} : Ker(F) &\mapsto Ker(F)\\
	r &\mapsto c \cdot r
	\end{align*}
	Note that $R$ is a commutative Artinian local ring, hence $\{\phi_{c}: c \in S_{0}\}$ is a set of commutative nilpotent linear maps on $Ker(F)$. Therefore we can choose a basis of $Ker(F)=\langle \mu_{1},...,\mu_{l} \rangle$ s.t. $\phi_{c} (\mu_{i}) \subseteq \langle \mu_{1},...,\mu_{i-1} \rangle $, for any $c \in S_{0}$. As  $S_{0}$ is a basis of $\mathfrak{m}$, (i) is proved.\par 
(ii) If $Ker(F)=0$ then $V^{(1)}=W \not =0$ from the definition	of $V^{(1)}$. If $Ker(F) \not =0$ then by (i) there exist a $\mu_{1} \not =0$ s.t. $\mu_{1} \cdot \mathfrak{m}=0$ and hence $\mu_{1} \in V^{(1)}$, concluding that $V^{(1)} \not =0$. 
\end{proof}
Now we use the correspondences given in Theorem \ref{correspondence _proj} and Theorem \ref{correspondence_hypersurface} to obtain the operation described in Theorem \ref{degeneration_construction}.
\begin{proof}[Proof of Theorem \ref{degeneration_construction}] Firstly note that $Q^{(1)}$ is a non-degenerate variety in $\textbf{L}^{(1)}$, hence it suffices to prove that there exist a linear space $\textbf{L}^{(1)}$ satisfying Theorem \ref{degeneration_construction} (i) and (ii).  In the following we assume $dim(V^{(1)})=m$ for some $m \leqslant n-1$. \par
(a) If $X$ is a hyperquadric, then we represent the action by $(R,W,F)$ with $x_{0} \in O$ s.t. $x_{0}=[1_{R}]$ and define $(V^{(1)},V_{(1)})$ as in Proposition \ref{identify_2}. Also by Lemma \ref{basis0} we have $0 \not =V^{(1)} \subsetneq W$.
	\begin{mycases}
		\case $V^{(1)} \cdot V^{(1)} \subseteq V^{(1)}$, then the induced action is an additive action on a projective space. From Lemma \ref{y0} we conclude that $V^{(1)}=V_{(1)}$.  \par
		In this case $R^{(1)}=V^{(1)} \oplus \langle 1_{R} \rangle$ is a well-defined subring of $R$.  Furthermore it can be easily seen that $R^{(1)}$ is a finite dimensional $\mathbb{K}$-local algebra with maximal ideal $\mathfrak{m}^{(1)}=V^{(1)}$. Then by HT-correspondence (Theorem  \ref{correspondence _proj}), the pair $(R^{(1)},V^{(1)})$ gives an additive action of $\mathbb{G}_{a}^{m}$ on the projective space $\mathbb{P}(R^{(1)})$ with open orbit $\mathbb{G}_{a}^{m} \cdot [1_{R}]$. On the other hand, by Remark \ref{action_algebrai}, the action is given through identifying $\mathfrak{g}(\mathbb{G}_{a}^{m})$ with $V^{(1)}$, hence from Proposition \ref{identify_2}.(ii) we conclude that up to equivalences the corresponding action is exactly induced by the action of $G^{(1)}$ on $R^{(1)}$. Thus the action of $G^{(1)}$ on $\mathbb{P}(R^{(1)})$ is an additive action on the projective space with open orbit $G^{(1)} \cdot [1_{R}]$, and $\mathbb{P}(R^{(1)}) \subsetneq \mathbb{P}(R)$ as $V^{(1)} \subsetneq W$. Above all we have found the subspace $\textbf{L}^{(1)}=\mathbb{P}(R^{(1)})=Q^{(1)}$ of $\mathbb{P}^{n+1}$ satisfying Theorem \ref{degeneration_construction} (i) and (ii):
		\[
		 \begin{tikzcd}
		G^{(1)} \times \mathbb{P}(R^{(1)}) \arrow{r} \arrow[hook]{d} & \mathbb{P}(R^{(1)}) \supseteq G^{(1)} \cdot [1_{R}] \arrow[hook]{d} \\%
		\mathbb{G}_{a}^{n} \times \mathbb{P}^{n+1} \arrow{r}& \mathbb{P}^{n+1}  \supseteq \mathbb{G}_{a}^{n} \cdot [1_{R}]
		\end{tikzcd}
		\]
		\par
		\case $V^{(1)} \cdot V^{(1)} \nsubseteq V^{(1)}$, then the induced action is an additive action on a hyperquadric. From Lemma \ref{y0} we conclude that $V^{(1)}\not =V_{(1)}$.  \par\par
		First we can choose a suitable $b_{0} \in \mathfrak{m}^2 \backslash W$ s.t. $V^{(1)} \cdot V^{(1)} \subseteq V^{(1)} \oplus \langle b_{0} \rangle $ and $b_{0} \cdot Ker(F)=0$. In fact, from $V^{(1)}\not =V_{(1)}$,  there exist $a,a' \in V^{(1)}$ with $F(a,a')=1$. Now we define $b_{0}=a\cdot a'$ then $b_{0} \in \mathfrak{m}^2 \backslash W$ and $b_{0} \cdot Ker(F)=0$ as $a \in V^{(1)}$. Moreover for any $c,c' \in V^{(1)}$:
		\begin{align*}
				 c \cdot c' =y_{0}(cc') \cdot b_{0} +(c\cdot c')_{|_{W}},
		\end{align*}
		  hence from $c \cdot c' \cdot Ker(F) =b_{0} \cdot Ker(F)=0$  we have $(c\cdot c')_{|_{W}} \in V^{(1)}$, concluding that $V^{(1)} \cdot V^{(1)} \subseteq V^{(1)} \oplus \langle b_{0} \rangle$.\par
		Now we set $R^{(1)}=V^{(1)} \oplus \langle b_{0} \rangle \oplus \langle 1_{R} \rangle$, $\mathfrak{m}^{(1)}= V^{(1)} \oplus \langle b_{0} \rangle $.  Then 
		\begin{align*}
		 &b_{0} \in (\mathfrak{m}^{(1)})^{2} \nsubseteq V^{(1)},\\
		 &b_{0} \cdot \mathfrak{m}^{(1)} \subseteq (\mathfrak{m}^{(1)})^3 \subseteq V^{(1)},
		\end{align*}
		 as $(\mathfrak{m}^{(1)})^3 \subseteq \mathfrak{m}^3 \subseteq W$ and $(\mathfrak{m}^{(1)})^3 \cdot Ker(F) =0$, where $\mathfrak{m}^3 \subseteq W$ follows from \cite[Theorem 5.1]{Arzhantsev2009HassettTschinkelCM} and the fact that $(R,W,F)$ reprensents an action on a hyperquadric.\par
		 Now it follows that $R^{(1)}$ is a finite dimensional local $\mathbb{K}$-algebra with maximal ideal $\mathfrak{m}^{(1)}= V^{(1)} \oplus \langle b_{0} \rangle $, $V^{(1)}$ is a hyperplane of $\mathfrak{m}^{(1)}$ generating the algebra $R^{(1)}$ such that $(\mathfrak{m}^{(1)})^{2} \nsubseteq V^{(1)}$ and $(\mathfrak{m}^{(1)})^3 \subseteq V^{(1)}$. Hence by Theorem \ref{correspondence_hypersurface} and \cite[Theorem 5.1]{Arzhantsev2009HassettTschinkelCM}, $(R^{(1)},\mathfrak{m}^{(1)},V^{(1)})$ corresponds to an additive action of $G_{a}^{m}$ on a hyperquadric $Q^{(1)}$ in $\mathbb{P}(R^{(1)})$ with open orbit $\mathbb{G}_{a}^{m} \cdot [1_{R}]$. Then similar to Case 1, by Remark \ref{action_algebrai} and Proposition \ref{identify_2}.(ii) we conclude that the corresponding action (up to equivalences) is exactly induced by the action of $G^{(1)}$ on $R^{(1)}$. Thus the action of $G^{(1)}$ on $\mathbb{P}(R^{(1)})$ induces an additive action on a hyperquadric $Q^{(1)}$ with the open orbit $O^{(1)}=G^{(1)} \cdot [1_{R}]$, and $\mathbb{P}(R^{(1)}) \subsetneq \mathbb{P}(R)$ as $V^{(1)} \subsetneq W$. Moreover in the more explicit correspondence Theorem \ref{correpondence} we can easily see the corresponding bilinear form $F^{(1)}$ is just $F_{|_{R^{(1)}}}$.\par
		Now $\mathbb{P}(R^{(1)})$ is already a subspace satisfying Theorem \ref{degeneration_construction} (i) and (ii):
		\[ \begin{tikzcd}
		G^{(1)} \times \mathbb{P}(R^{(1)}) \arrow{r} \arrow[hook]{d} & \mathbb{P}(R^{(1)}) \supseteq Q^{(1)} \supseteq O^{(1)}= G^{(1)} \cdot [1_{R}]   \arrow[hook]{d} \\%
		\mathbb{G}_{a}^{n} \times \mathbb{P}^{n+1} \arrow{r}& \mathbb{P}^{n+1}  \supseteq Q \supseteq O= \mathbb{G}_{a}^{n} \cdot [1_{R}]
		\end{tikzcd}
		\]
	\end{mycases}
(b) If $X$ is a projective space, following Theorem \ref{correspondence _proj}, we represent the action $(\mathbb{G}_{a}^{n},\mathbb{P}^{n})$ by  a pair $(R,\mathfrak{m})$, where $x_{0}=[1_{R}]$. We first show that $K(X)=\mathbb{P}(\mathfrak{m})$. In fact, for any $[\alpha]$ in the open orbit we have $\alpha$ is invertible by Remark 2.2. Conversely, for any invertible element $r \in R$ we have $dim(\mathbb{G}_{a}^{n} \cdot [r])=dim(\mathfrak{g} \cdot r)=dim(\mathfrak{m} \cdot r)=dim(\mathfrak{m})=n$, concluding that $[r]$ lies in the open orbit. Now we define $V^{(1)}=\{\alpha \in \mathfrak{m}:\alpha \cdot \mathfrak{m}=0\}$, then $Fix(X)=\mathbb{P}(V^{(1)})$. Since $K(X) \nsubseteq Fix(X)$ by assumption of Theorem 1.6, we have $V^{(1)} \subsetneq \mathfrak{m}$. Moreover as elements in $\mathfrak{m}$ are nilpotent, we conclude that $V^{(1)} \not =0$ by a similar discussion as that in Lemma 3.2.\par
Now we consider $R^{(1)}=V^{(1)} \oplus \langle 1_{R} \rangle$ then similar to Case 1 of (a), $\mathbb{P}(R^{(1)})$ is $G^{(1)}$-stable and the induced action is an additive action on a projective space with open orbit $G^{(1)} \cdot [1_{R}]$, and $\mathbb{P}(R^{(1)}) \subsetneq \mathbb{P}(R)$ as $V^{(1)} \subsetneq \mathfrak{m}$. Thus $\mathbb{P}(R^{(1)})$ is already a subspace satisfying Theorem \ref{degeneration_construction} (i) and (ii). 
\end{proof}
 Combining the above proof with Proposition \ref{identify_2}, we have the following.
\begin{proposition}\label{output_condi}
Given an additive action on a hyperquadric $Q$ with unfixed singularities, we represent the operation obtained in Theorem \ref{degeneration_construction} by $(R,W,F) \mapsto (R^{(1)},V^{(1)},F^{(1)})$, then:\par
(i)	$Q^{(1)}$ is a projective space  if and only if $V^{(1)} \cdot V^{(1)} \subseteq V^{(1)}$ if and only if $V^{(1)}=V_{(1)}$.\par
(ii) $Sing(Q) \nsubseteq K(Q^{(1)})$ if and only if $ Ker(F) \nsubseteq V_{(1)}$, $Sing(Q)=K(Q^{(1)})$ if and only if $ Ker(F)=V_{(1)}$.\par
(iii) the operation is effective if and only if $Ker(F) \subsetneqq V_{(1)}$.
\end{proposition}
\begin{proof}
(i)  By part (a) in the proof of Theorem 1.6, it suffices to show $V^{(1)} \cdot V^{(1)} \subseteq V^{(1)}$ if $V^{(1)}=V_{(1)}$. In this case, for any $a,a' \in V^{(1)}$ we have $F(a,a')=0$, hence $aa' \in W$ by Lemma \ref{y0} and $aa' \cdot Ker(F)=0$ by the definition of $V^{(1)}$, concluding that $V^{(1)} \cdot V^{(1)} \subseteq V^{(1)}$.\par
(ii) and (iii). By our definition of effective operation \ref{effective} and $Sing(Q)=\mathbb{P}(Ker(F))$, it suffices to show $K(Q^{(1)})=\mathbb{P}(V_{(1)})$. If $Q^{(1)}$ is a projective space then from part (b) in the proof of Theorem 1.6, for the action on $Q^{(1)}$ represented by $(R^{(1)},\mathfrak{m}^{(1)})$ we have $K(Q^{(1)})=\mathbb{P}(\mathfrak{m}^{(1)})=\mathbb{P}({V}^{(1)})=\mathbb{P}(V_{(1)})$ since in this case $\mathfrak{m}^{(1)}=\mathbb{P}(V^{(1)})=V_{(1)}$ by Case 1 of part (a) in the proof of Theorem 1.6. If $Q^{(1)}$ is a hyperquadric, then $K(Q^{(1)})=Sing(Q^{(1)})=\mathbb{P}(Ker(F^{(1)}))$ and $Ker(F^{(1)})=Ker(F^{(1)}_{|V^{(1)}}))$ by Lemma 2.9. Finally as $F^{(1)}=F_{|R^{(1)}}$, we have $Ker(F^{(1)})=Ker(F_{| V^{(1)}})=V_{(1)}$ by the definition of $V_{(1)}$, concluding the proof.
\end{proof}
 \subsection{Unfixed singularities and vanishing bilinear form}
Our main result of this section is the following.
\begin{proposition}\label{mainprosec2}
    For an additive action on a hyperquadric $Q$ of corank 2 represented by $(R,W,F)$. If $Sing(Q) \nsubseteq Fix(Q)$ and $dim(Q) \geqslant 5$, then for the operation obtained in Theorem \ref{degeneration_construction} we have:\par
    (i) $codim(Q^{(1)},Q)=codim(V^{(1)},W)=1$. \par
    (ii) there exist  $b_{0} \in \mathfrak{m}^{2} \backslash W$ with $F(1,b_{0})=1$ and a $\mathbb{K}$-basis of $Ker(F)$, $\mu_{1},\mu_{2}$, such that:
    \begin{align*}
    b_{0} \cdot \mathfrak{m}=\mu_{1} \cdot \mathfrak{m} =0 \\
    \mu_{2} \cdot \mathfrak{m} \subseteq \langle \mu_{1} \rangle \\
    V^{(1)} \cdot \mathfrak{m} \subseteq \langle \mu_{1},b_{0} \rangle 
    \end{align*}
    (iii) if the operation is not effective, i.e., $Ker(F)=V_{(1)}$ or $Ker(F) \nsubseteq V_{(1)}$, then we can normalize the algebraic structrue of $(R,W,F)$. (see Lemma 3.6 and Lemma 3.9 for details). 
\end{proposition}
First applying Lemma \ref{basis0} we have the following:
\begin{lemma}\label{codim^1}
(i) there exist suitable basis of $Ker(F)$, $\mu_{1},\mu_{2}$, s.t. $\mu_{1} \cdot \mathfrak{m}=0$ and $\mu_{2} \cdot \mathfrak{m} \subseteq \langle \mu_{1} \rangle $.
(ii) $codim(V^{(1)},W)=1$.
\end{lemma}
\begin{proof}
(i) Applying Lemma \ref{basis0} when $l=2$.\par
(ii) For any $r \in W$ we have $r \cdot \mu_{2} =\lambda_{r} \cdot \mu_{1}$ for some $\lambda_{r} \in \mathbb{K}$, this induces a linear form on $W$:
\begin{align*}
    \Phi: W &\mapsto \mathbb{K} \\
           \alpha &\mapsto \lambda_{\alpha}
\end{align*}
Hence we have $V^{(1)}=Ker(\Phi)$ and $codim(V^{(1)},W)=1$.
\end{proof}
From now on we always choose a basis of $Ker(F)$ satisying Lemma \ref{codim^1}.(i).\par
We prove Proposition \ref{mainprosec2} through a case-by-case argument on analyzing the relation between $Ker(F)$ and $V_{(1)}$. More precisely, we separate it into the following cases.\par
1. $Sing(Q) \subseteq K(Q^{(1)})$, i.e., $Ker(F) \subseteq V_{(1)}$. In this case we have nice inclusions between subspaces:
   $Ker(F) \subseteq V_{(1)} \subseteq V^{(1)} \subseteq W$, for which we furtherly consider two subcases:\par
   (1.a). $Sing(Q)= K(Q^{(1)})$, i.e., $Ker(F)=V_{(1)}$. In this subcase, we can normalize the algebraic structure of $(R,W,F)$.  \par
   (1.b). The operation on $(\mathbb{G}_{a}^{n},Q)$ is effective, i.e., $Ker(F) \subsetneq V_{(1)}$. In this subcase, it remains to determine the multiplication between elements in $V_{(1)}$ and $V^{(1)}$, which leads to our definition of $(V^{(2)},V_{(2)})$ and further discussions in Section 4. 
    \par
2. $Sing(Q) \nsubseteq K(Q^{(1)})$, i.e, $Ker(F) \nsubseteq V_{(1)}$. In this case, we can normalize the algebraic structure of $(R,W,F)$.
\subsubsection{$Ker(F)=V_{(1)}$} 
Recall $V_{(1)}=Ker(F|_{V^{(1)}})$ and $Ker(F|_{W})=Ker(F)$ by Lemma \ref{ker_kerr}, hence we can  have a decomposition of $W$ as follows:
\begin{align}\label{deco1}
W=Ker(F) \oplus \langle e_{1},...,e_{t} \rangle \oplus \langle e_{t+1} \rangle,
\end{align}
where $t \geqslant 2$, $e_{i} \in V^{(1)}$ for $1 \leqslant i \leqslant t$, $e_{t+1} \in W \backslash V^{(1)}$ and $F(e_{i},e_{j})=\delta_{i,j}$.  Then we can furtherly choose a suitable $b_{0}$ and $e_{i},e_{t+1}$ to give a normalization of this case:
\begin{lemma}\label{orlemma}
	If $Ker(F)=V_{(1)}$, then let $b_{0}=e_{1}^2$ we have: \par
	(i) $b_{0} \in \mathfrak{m}^2 \backslash W$ and $b_{0} \cdot W =b_{0} \cdot \mathfrak{m}=0$,  $ V^{(1)} \cdot \mathfrak{m} \subseteq \langle \mu_{1},b_{0} \rangle $.\par
	(ii) one can choose suitable $e_{i},e_{t+1}$ such that 
	\begin{align*}
	e_{t+1} \cdot e_{i}=0 ,	e_{t+1} \cdot \mu_{2}=\mu_{1}, e_{t+1}^{2}=b_{0}+\delta \cdot \mu_{2}	,
	\end{align*}
	where $1 \leqslant i \leqslant t$, $\delta =1$ if $dim(\mathfrak{m}^{2})=3$ and $\delta =0$ if $dim(\mathfrak{m}^{2})=2$.
\end{lemma}
	\begin{proof}
	(i) As $F(e_{1},e_{1})=1 \not =0$ we have $b_{0}=e_{1}^2 \in \mathfrak{m}^2 \backslash W$ from Lemma \ref{y0}. By formula (\ref{multi}) we can describe the multiplications in $\mathfrak{m}$ as follows:
	\begin{equation}\label{multiply}
	aa'=F(a,a')\cdot b_{0} + V_{1}(a,a')\cdot \mu_{1} +V_{2}(a,a')\cdot \mu_{2}.
	\end{equation}
	Note that from $e_{1} \in V^{(1)}$ we have $b_{0}\cdot Ker(F)=0$, hence to show $b_{0} \cdot W =b_{0} \cdot \mathfrak{m}=0$ it suffices to check $b_{0} \cdot e_{i}=0 $ for $1 \leqslant i \leqslant t+1$.\par
	 For any $1\leqslant  i \leqslant t$, we choose some $j \not = i$. Then from $e_{i},e_{j} \in V^{(1)}$ we have:
	\begin{equation}\label{b0*w=0}
	\begin{split}
	b_{0} \cdot e_{i}=&(e_{j}^2-V_{1}(e_{j},e_{j})\cdot \mu_{1}-V_{2}(e_{j},e_{j}) \cdot \mu_{2}) \cdot e_{i}
	=e_{j}^2 \cdot e_{i}\\
	=&e_{j} \cdot (\delta_{i,j} \cdot b_{0} +V_{1}(e_{i},e_{j}) \cdot \mu_{1} + V_{2}(e_{i},e_{j}) \cdot \mu_{2})=0.
	\end{split}
	\end{equation}
	For $e_{t+1}$ we have :
	\begin{align*}
	b_{0} \cdot e_{t+1}=e_{1}^2 \cdot e_{t+1} 
	=e_{1} \cdot (\delta_{1,t+1} \cdot b_{0} +V_{1}(e_{1},e_{t+1}) \cdot \mu_{1} + V_{2}(e_{1},e_{t+1}) \cdot \mu_{2})=0	.
	\end{align*}
Now for any $a \in V^{(1)}$ and any $a' \in W$, by multiplying $e_{t+1}$ to both sides of equation (\ref{multiply}) we have:
\begin{equation*}
\begin{split}
LHS=&e_{t+1} \cdot a \cdot a'=a \cdot (F(e_{t+1},a') \cdot b_{0} +V_{1}(e_{t+1},a') \cdot \mu_{1}+V_{2}(e_{t+1},a') \cdot \mu_{2})=0.\\
RHS=&e_{t+1} \cdot (  -F(a,a')\cdot b_{0} + V_{1}(a,a')\cdot \mu_{1} +V_{2}(a,a')\cdot \mu_{2})
 = \lambda_{t+1} \cdot V_{2}(a,a') \cdot \mu_{1}
 \end{split}
\end{equation*}
where $e_{t+1} \cdot \mu_{2} = \lambda_{t+1} \cdot \mu_{1}$ with $\lambda_{t+1} \not =0$ by $e_{t+1} \in W \backslash V^{(1)}$. Hence form $LHS=RHS$ we have $V_{2}(a,a')=0$. Thus $V^{(1)} \cdot W \subseteq \langle b_{0}, \mu_{1} \rangle$. Since $W \cdot \langle \mu_{1},b_{0} \rangle=0$ by arguments above, $V^{(1)} \cdot W^{(k)}=0$ for all $k \geqslant 2$. Since $\mathfrak{m}$ is generated by $W$, we conclude that $V^{(1)} \subseteq \langle b_{0},\mu_{1} \rangle$.\par
(ii) Firstly as $F(e_{t+1},e_{i})=0$ for $1 \leqslant i \leqslant t$ and from (i) we have $e_{t+1} \cdot e_{i} \in \langle \mu_{1} \rangle $. Thus if we replace $e_{i}$ by $e_{i}-\lambda_{t+1}^{-1} V_{1}(e_{i},e_{t+1}) \cdot \mu_{2}$ then $e_{t+1} \cdot e_{i}=0$ and we still have $F(e_{i},e_{j})=\delta_{i,j}$. Furtherly by (3.2) we have:
\begin{align*}
e_{t+1}^{2}=b_{0} + V_{1}(e_{t+1},e_{t+1})\cdot \mu_{1} +V_{2}(e_{t+1},e_{t+1})\cdot \mu_{2}.
\end{align*}
then we can replace $e_{t+1}$ by $e_{t+1}-\frac{V_{1}(e_{t+1},e_{t+1})}{2 \lambda_{t+1}} \cdot \mu_{2}$ to make $V_{1}(e_{t+1},e_{t+1})=0$. Note that this will not affect the multiplication of $e_{t+1}$ and $e_{i}$ for $i \leqslant t$. Then by (i) and  Lemma \ref{basic_multi} we conclude that  $V_{2}(e_{t+1},e_{t+1}) \not =0 $ if and only if $ dim(\mathfrak{m}^{2})=3$. Now if $V_{2}(e_{t+1},e_{t+1}) \not =0$, we replace $\mu_{2}$ by $V_{2}(e_{t+1},e_{t+1}) \cdot \mu_{2}$ to make $e_{t+1}^{2}=b_{0}+\delta \cdot \mu_{2}$ and then replace $\mu_{1}$ by $e_{t+1} \cdot \mu_{2}$ to make $e_{t+1} \cdot \mu_{2}=\mu_{1}$.
\end{proof}
\subsubsection{$Ker(F) \subsetneqq V_{(1)}$} In this subcase we start with the following observation.
\begin{observation}\label{ori_obser}
	$codim(Ker(F),V_{(1)})=1$.
\end{observation}
\begin{proof}
As $Ker(F) \subseteq V_{(1)} \subseteq V^{(1)} \subseteq W$ and $Ker(F|_{W})=Ker(F)$, we have a natural injective linear map:
	\begin{equation*}
		\begin{split}
	V_{(1)} / Ker(F) &\xmapsto{\sigma} (W / V^{(1)})^* \\
	 \overline{\alpha} &\mapsto  \sigma(\overline{\alpha}): \overline{\beta} \rightarrow F(\alpha,\beta)
	 \end{split}
	\end{equation*}
hence $codim(Ker(F),V_{(1)}) \leqslant codim(V^{(1)},W)=1$, concluding the proof.
\end{proof}
Note that by the assumption of $dim(W) \geqslant 5$ we have $codim(V_{(1)},V^{(1)}) \geqslant 1$. And by $Ker(F|_{W})=Ker(F)$ we have a decomposition of $W$ in this subcase:
\begin{equation}
W=\overbrace{\underbrace{\overbrace{\langle \mu_{1},\mu_{2}\rangle}^{Ker(F)} \oplus \langle g_{1}  \rangle}_{V_{(1)}} \oplus \langle e_{1},e_{2},...,e_{t} \rangle}^{V^{(1)}} \oplus \langle f_{1} \rangle \, \, \, \, \,(t \geqslant 1)
\end{equation}
 We now can find a suitable $b_{0}$.
\begin{lemma}
	Let $b_{0}=e_{1}^2$ then $b_{0} \in \mathfrak{m}^{2} \backslash W$ and $b_{0} \cdot W =b_{0} \cdot \mathfrak{m}=0$,  $ V^{(1)} \cdot \mathfrak{m} \subseteq \langle \mu_{1},b_{0} \rangle $.
\end{lemma}
\begin{proof}
    First we check that $b_{0} \cdot W =b_{0} \cdot \mathfrak{m}=0$. For $b_{0} \cdot g_{1}=0$:
	\begin{align*}
	b_{0} \cdot g_{1} =e_{1}^2 \cdot g_{1}
	 =e_{1} \cdot (F(e_{1},g_{1}) \cdot b_{0}+V_{1}(e_{1},g_{1}) \cdot \mu_{1} + V_{2}(e_{1},g_{1}) \cdot \mu_{2})=0,
	\end{align*}
	where the last equation follows from $F(e_{1},g_{1})=0$ and $e_{1} \in V^{(1)}$.\par
    To show $b_{0} \cdot e_{i}=0$ for any $1 \leqslant i \leqslant t$. Firstly note that if $t \geqslant 2$ then we can prove it by using the same computation as (\ref{b0*w=0}).\par
     Now we assume $t=1$. As $g_{1} \in V_{(1)} \backslash Ker(F)$ we can assume $F(g_{1},f_{1})=1$ moreover we can assume $F(f_{1},e_{1})=0$ up to replacing $e_{1}$ by $e_{1}-F(e_{1},f_{1}) \cdot g_{1}.$ Then the calculation of $b_{0} \cdot e_{1}$ follows:
     \begin{align*}
     b_{0} \cdot e_{1}&=(f_{1} \cdot g_{1}-V_{1}(f_{1},g_{1}) \cdot \mu_{1}-V_{2}(f_{1},g_{1}) \cdot \mu_{2}) \cdot e_{1}= g_{1} \cdot f_{1} \cdot  e_{1}\\
     &=g_{1} \cdot (F(f_{1},e_{1}) \cdot b_{0}+V_{1}(f_{1},e_{1}) \cdot \mu_{1}+V_{2}(f_{1},e_{1}) \cdot \mu_{2})=0,
     \end{align*}
     
where the last equation follows from $g_{1} \in V^{(1)}$. Then the calculation of $b_{0} \cdot f_{1}$ follows:
\begin{align*}
b_{0} \cdot f_{1}&=e_{1}^2 \cdot f_{1}
=e_{1} \cdot (F(e_{1},f_{1}) \cdot b_{0}+V_{1}(e_{1},f_{1}) \cdot \mu_{1} + V_{2}(e_{1},f_{1}) \cdot \mu_{2})=0. 
\end{align*}
Finally we conclude that $V^{(1)} \cdot W \subseteq \langle b_{0},\mu_{1} \rangle$ by multiplyng $f_{1}$ to both sides of the formula (3.2).
\end{proof}
\subsubsection{ $Ker(F) \cdot Ker(F) \not =0$} By Lemma \ref{codim^1} we have $\mu_{1} \cdot Ker(F)=0 $ and $\mu_{2} \cdot W \subseteq \langle \mu_{1} \rangle$, hence we can assume $\mu_{2}^2=\mu_{1}$. Now we have a decomposition of $W$:
\begin{equation}
W=\langle \mu_{1},\mu_{2} \rangle \oplus \langle e_{1},...,e_{t} \rangle ,
\end{equation}
with $F(e_{i},e_{j})= \delta_{i,j}$. Moreover for any $e_{i}$ with $e_{i} \cdot \mu_{2}=\lambda_{i} \cdot \mu_{1}$ we can replace $e_{i}$ by $e_{i}-\lambda_{i} \cdot \mu_{2}$ to make $e_{i} \cdot \mu_{2}=0$, which does not affect the value of $F(e_{i},e_{j})$ as $\mu_{2} \in Ker(F)$. Then $V^{(1)}= \langle \mu_{1},e_{1},...,e_{t} \rangle$ and we can find suitable $b_{0}$ as before, which also gives a normalization of this subcase:
\begin{lemma}
Let $b_{0}=e_{1}^2$ then $b_{0} \in \mathfrak{m} \backslash W$ and \par
(i) $b_{0} \cdot W =b_{0} \cdot \mathfrak{m}=0$,  $ V^{(1)} \cdot \mathfrak{m} \subseteq \langle \mu_{1},b_{0} \rangle $.\par
(ii) $\mu_{2}^{2}=\mu_{1}$, $e_{i} \cdot \mu_{2}=e_{i} \cdot \mu_{1}=0$.
\end{lemma}
\begin{proof}
It suffices to prove (i). First we note that we can use the same method in Lemma \ref{orlemma} to show $b_{0} \cdot W=0$. Then it suffices to prove $V^{(1)} \cdot \mathfrak{m} \subseteq \langle \mu_{1},b_{0} \rangle$. For any $a \in V^{(1)} $, $a' \in W$ equation (\ref{multiply}) still holds and in this case we multiply it by $\mu_{2}$:
	\begin{equation*}
	\begin{split}
	LHS=& \mu_{2} \cdot a \cdot a'=0.\\
	RHS=& \mu_{2} \cdot (F(a,a')\cdot b_{0}+V_{1}(a,a') \cdot \mu_{1}+V_{2}(a,a') \cdot \mu_{2})=V_{2}(a,a') \cdot \mu_{1}.
	\end{split}
	\end{equation*}
	Then from $LHS=RHS$ we have $V_{2}(a,a')=0$, concluding the proof.
\end{proof}

\section{Classification of actions with unfixed singularities}
\subsection{Classification of actions with unfixed singularities (I): $dim(Q) \geqslant 5$ }
In this and next subsections we always consider additive actions on hyperquadrics of corank two with unfixed singularities. Firstly we give the algebraic version of the flow chart, which induces an $algebraic$ $structure$ $sequence$ for a given triple $(R,W,F)$. Then by analyzing the sequence we normalize the structure of $(R,W,F)$. Finally we show the uniqueness of the normalized structure up to equivalences.
\subsubsection{Algebraic version of the flow chart}
Recall in the proof of Theorem \ref{degeneration_construction}, we have represented an operation $(\mathbb{G}_{a}^{n},Q,\mathbb{P}^{m}) \rightarrow (G^{(1)},Q^{(1)},L^{(1)})$ by $(R,W,F) \rightarrow (R^{(1)},V^{(1)},F^{(1)})$ or $(R,W,F) \rightarrow (R^{(1)},V^{(1)})$. In Proposition \ref{output_condi}, we also gave the algebraic criterion for the output condition in the flow chart. Thus the algebraic version of the flow chart naturally arises as the following:
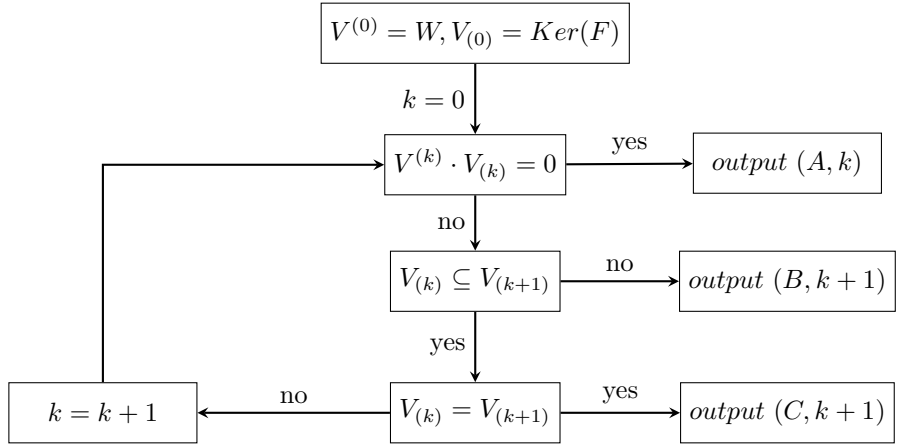
\begin{figure}[H]\label{alg_ver_diag}
	\centering
	\tikzstyle{io} = [trapezium, trapezium left angle=70, trapezium right angle=110, minimum width=2.5cm, minimum height=0.8cm, text centered, draw=black, fill=blue!30]
	\tikzstyle{process} = [rectangle, minimum width=2.5cm, minimum height=0.8cm, text centered, draw=black]
	\tikzstyle{decision} = [rectangle, minimum width=1.0cm, minimum height=0.8cm, text centered, draw=black]
	\tikzstyle{arrow} = [thick,->,>=stealth]
	\begin{tikzpicture}[node distance=1.15cm]   
	\node (pro0) [process] {$V^{(0)}=W,V_{(0)}=Ker(F)$};
	\node (dec1) [decision,below of= pro0,yshift=-0.6cm] {$V^{(k)} \cdot V_{(k)}=0 $}; 
	\node (pro1) [process,right of =dec1,xshift=3.0cm,yshift=0.6] {$output$ $(A,k)$};
	\node (dec2) [decision,below of = dec1,yshift=-0.4cm] {$V_{(k)} \subseteq V_{(k+1)}$};
	\node (pro2)[process,right of = dec2,xshift=3.0cm,yshift=0.0cm] {$output$ $(B,k+1)$};
	\node (dec3) [decision,below of =dec2,yshift=-0.6cm] {$V_{(k)}=V_{(k+1)}$};
	\node (pro4) [process,right of =dec3,xshift=3.0cm,yshift=0.0cm] {$output$ $(C,k+1)$};
	\node (pro5) [process,left of =dec3,xshift=-3.8cm,yshift=0.0cm] {$k=k+1$};
	
	\draw [arrow] (pro0) -- node[anchor=east]{$k=0$}(dec1);
	\draw [arrow] (dec1) -- node [anchor=south] {yes} (pro1);
	\draw [arrow] (dec1) -- node [anchor=east] {no} (dec2) ;
	\draw [arrow] (dec2) -- node [anchor=south]{no} (pro2);
	\draw [arrow] (dec2) -- node [anchor=east] {yes} (dec3);
	\draw [arrow] (dec3) -- node [anchor=south] {yes} (pro4);
	\draw [arrow] (dec3) -- node [anchor=south]  {no}  (pro5);
	\draw [arrow] (pro5) |- (dec1) ;
	\end{tikzpicture}
\end{figure}
\noindent where for any $(V^{(k)},V_{(k)})$ if $V^{(k)} \cdot V_{(k)} \not =0$ we furtherly define:
\begin{equation}
\begin{split}
V^{(k+1)}&=\{\alpha \in V^{(k)}: \alpha \cdot V_{(k)}=0\}\\
V_{(k+1)}&=Ker(F_{|_{V^{(k+1)}}}) 
\end{split}
\end{equation}
and we represent the final output by $(x,s,V^{(s)},V_{(s)})$, where for a output $(x,t)$ we set $s=t-1$ if $x \in \{B,C\}$ and $s=t$ if $x=A$.\par
Then for the final output we obtain an $algebraic$ $structure$ $sequence$ as follows:
\begin{align}
Ker(F) = V_{(0)} \subseteq...\subseteq V_{(s)}\subseteq V^{(s)}  \subseteq ... \subseteq V^{(0)}=W,
\end{align}
where $ V^{(k)} \cdot V_{(k-1)}=0$ for $1 \leqslant  k \leqslant s$.\par
For the sequence, our first step is to  generalize Proposition 3.4 (i) and Observation \ref{ori_obser} to the following.
\begin{proposition}\label{codim_1}
	For an algebraic structure sequence: $\{(V^{(k)},V_{(k)}): 0 \leqslant k \leqslant s\}$:\par
	$A_{(k)}:$ if $V^{(k+1)} \subsetneqq V^{(k)}$ then $codim(V^{(k+1)},V^{(k)})=1$;\par
	$B_{(k)}:$ if $V_{(k)} \subsetneqq V_{(k+1)}$ then $codim(V_{(k)},V_{(k+1)})=1$.
\end{proposition}
\begin{proof}
	Firstly note that if $V^{(k+1)} \subsetneqq V^{(k)}$ then $V^{(i+1)} \subsetneqq V^{(i)}$ for any $i \leqslant k-1$, similarly if $V_{(k)} \subsetneqq V_{(k+1)}$ then $V_{(i)} \subsetneqq V_{(i+1)}$ for any $i \leqslant k-1$. Hence we can prove $A_{(k)}$and $B_{(k)}$ by induction on $k$.\par
	For $k=0$, $A_{(0)}$ follows from Lemma \ref{codim^1} and $B_{(0)}$ follows from Observation \ref{ori_obser}. Now assuming $A_{(k-1)}$ and $B_{(k-1)}$ is true for some $k \geqslant 1$, then for a given $(V_{(k)},V^{(k)})$ in the process we already have $V_{(k-1)} \subsetneqq V_{(k)} \subseteq V^{(k)} \subsetneqq V^{(k-1)}$ with $V^{(k)} \cdot V_{(k-1)}=0$ and  $codim(V_{(k-1)},V_{(k)})=1$ by induction. Now since $(R,W,F)$ represents an action on a hyperquadric of corank two with unfixed singularities and $V_{(k)} \subseteq V_{(0)}=W, V^{(k)} \subseteq V^{(1)}$, we have $V^{(k)} \cdot V_{(k)} \subseteq V^{(1)} \cdot W \subseteq  \langle \mu_{1}\rangle $ by Proposition \ref{mainprosec2} (ii). Hence by the definition of $V^{(k+1)}$ we conclude that $codim(V^{(k+1)},V^{(k)}) \leqslant 1 $, implying $A_{(k)}$. \par
	Now if $V_{(k)} \subsetneqq V_{(k+1)}$ then from the process we already have $V^{(k+1)} \subsetneqq V^{(k)}$ and $A_{(k)}$ holds. Moreover we have the chain $V_{(k)} \subsetneqq V_{(k+1)} \subseteq V^{(k+1)} \subsetneqq V^{(k)}$, which induces an injective map:
	\begin{equation*}
	\begin{split}
	V_{(k+1)} / V_{(k)} &\xmapsto{\sigma_{k}} (V^{(k)} / V^{(k+1)})^* \\
	\overline{\alpha} &\mapsto  \sigma(\overline{\alpha}): \overline{\beta} \mapsto F(\alpha,\beta)
	\end{split}
	\end{equation*}
	It follows that $codim(V_{(k)},V_{(k+1)}) \leqslant codim(V^{(k+1)},V^{(k)})=1$, implying $B_{(k)}$.\\
\end{proof}

\subsubsection{Normalization} 
In this subsection we normalize the structure of $(R,W,F)$ by analyzing the algebraic structure sequence case by case.\par 
In the following, we always start with a $b_{0} \in \mathfrak{m}^{2} \backslash W$ and a basis of $Ker(F)$, $\mu_{1},\mu_{2}$, satisfying Proposition \ref{mainprosec2}. We furtherly define $V_{(-1)}=\langle \mu_{1} \rangle $.\\
\begin{mycases}
\case $x=A$. In this case the sequence becomes
	\begin{align*}
	V_{(-1)} \subsetneqq Ker(F)=V_{(0)} \subsetneqq...\subsetneqq V_{(s)}\subseteq V^{(s)}  \subsetneqq ... \subsetneqq V^{(0)}=W,
	\end{align*}
with $V^{(k)} \cdot V_{(k-1)}=V^{(s)} \cdot V_{(s)}=0$ for $ 1 \leqslant k \leqslant s$ (here $s \geqslant 1$ as we assume there exist unfixed singular points). Then we have the following normalization.
\begin{lemma}\label{norm_A}
	(i) If $V^{(s)} \not = V_{(s)}$ then there exist $f_{i} \in  V^{(i-1)} \backslash V^{(i)}$, $g_{i} \in V_{(i)} \backslash V_{(i-1)}$ for $1 \leqslant i \leqslant s$, $g_{0} \doteq \mu_{2}$ and $\{e_{k}: 1\leqslant k \leqslant p\} \subseteq V^{(s)} \backslash V_{(s)}$ such that 
	\begin{equation}
	\begin{split}
	V^{(s)}&=V_{(s)}\oplus \langle e_{1},...,e_{p} \rangle, \\
	e_{k} \cdot e_{l} &=\delta_{k,l} \cdot b_{0}+V_{1}(e_{k},e_{l}) \cdot \mu_{1},  \\
	\end{split}
	\end{equation}
	and
	\begin{equation}
     \begin{split}
	e_{k} \cdot f_{i}&=e_{k} \cdot g_{i}=f_{i} \cdot f_{j}=f_{v} \cdot g_{v'}=0, \\
	f_{i} \cdot g_{i}&=b_{0}, \,	f_{i} \cdot g_{i-1}=\mu_{1}, \, f_{1}^2=\delta \cdot \mu_{2},
	\end{split}
	\end{equation} 
	for $1 \leqslant i  \leqslant s$, $1 \leqslant k ,l \leqslant p $, $ v-v' \not \in \{0,1\}$, $ 2\leqslant j \leqslant s$ when $s \geqslant 2$,
	\[
	\delta=
	\begin{cases*}
	0 &\text{if $dim(\mathfrak{m}^2)=2$;}\\
	1 & \text{if $dim(\mathfrak{m}^2)=3$}
	\end{cases*}
	\] 
	and  the matrix $\Lambda=(V_{1}(e_{k},e_{l}):1 \leqslant k , l \leqslant p )$ is of the canonical form (see (\ref{canonical_form}) below).
	\par
	(ii) If $V^{(s)}=V_{(s)}$ then there exist $f_{i} \in  V^{(i-1)} \backslash V^{(i)}$, $g_{i} \in V_{(i)} \backslash V_{(i-1)}$ for $1 \leqslant i \leqslant s$, $g_{0} \doteq \mu_{2}$ such that
	\begin{align*}
	f_{i} \cdot f_{j}=f_{v} \cdot g_{v'}=0, \, f_{i} \cdot g_{i}&=b_{0}, \,	f_{i} \cdot g_{i-1}=\mu_{1}, \, f_{1}^2=\delta \cdot \mu_{2},
	\end{align*} 
	for $1 \leqslant i \leqslant s$, $v-v' \not \in \{0,1\}$, $ 2\leqslant j \leqslant s$ when $s \geqslant 2$, and
$\delta$ is the same as in (i).
\end{lemma}
\begin{proof}
Recall by Proposition \ref{mainprosec2} (ii) we always have $V^{(0)} \cdot V^{(0)} \subseteq \langle \mu_{1},\mu_{2},b_{0} \rangle $ and $V^{(1)} \cdot V^{(0)} \subseteq \langle \mu_{1},b_{0} \rangle$. Hence if choosing any nonzero $f_{i} \in V^{(i-1)} \backslash V^{(i)}$ and any nonzero $h_{i-1} \in V_{(i-1)} \backslash V_{(i-2)}$ then we can have $f_{i} \cdot h_{i-1} =c \cdot \mu_{1}$ for some nonzero $c$ as $V^{(i-1)} \cdot V_{(i-2)}=0$, $V^{(i-1)} \cdot V_{(i-1)}\not=0$  and $codim(V_{(i-2)},V_{(i-1)})=1$. Moreover choosing any nonzero $g_{i} \in V_{(i)} \backslash V_{(i-1)}$ we have $F(f_{i},g_{i}) \not =0$ from the definition of $V_{(i)}$. \par
(i) If $V_{(s)} \not =V^{(s)}$ we can choose $e_{k}$'s satisfying (4.3), i.e., $F(e_{k},e_{l})=\delta_{k,l}$. Then we find $f_{i},g_{i}$ inductively. For $i=s$ we first choose $f_{s} \in V^{(s-1)}\backslash V^{(s)}$, $g_{s} \in V_{(s)} \backslash V_{(s-1)}$ and  $h_{s-1} \in V_{(s-1)} \backslash V_{(s-2)}$ s.t. $f_{s} \cdot h_{s-1} =\mu_{1}$, $F(f_{s},g_{s})=1$ and $F(f_{s},f_{s})=0$. Then for the multiplications: 
\begin{equation*}
\begin{split}
 f_{s} \cdot g_{s}&=b_{0}+V_{1}(f_{s},g_{s}) \cdot \mu_{1},\\
  f_{s}\cdot e_{k}&=F(f_{s},e_{k}) \cdot b_{0} +V_{1}(f_{s},e_{k}) \cdot \mu_{1},\\
 f_{s}^2&=V_{1}(f_{s},f_{s}) \cdot \mu_{1}+V_{2}(f_{s},f_{s}) \cdot \mu_{2},
 \end{split}
\end{equation*}
we can normalize them through the following steps:
\begin{align*}
 g_{s} &\mapsto g_{s}-V_{1}(f_{s},g_{s}) \cdot h_{s-1} &&\text{to make $g_{s} \cdot f_{s}=b_{0}$} \\
 e_{k} & \mapsto e_{k}+F(f_{s},e_{k}) \cdot g_{s}-V_{1}(f_{s},e_{k}) \cdot h_{s-1} &&\text{to make $f_{s} \cdot e_{k}=0$}\\
 f_{s} &\mapsto  f_{s}-(V_{1}(f_{s},f_{s})/2) \cdot h_{s-1} &&\text{to make $f_{s}^2=
 	\begin{cases*}
 	0 &\text{if $s \geqslant 2$}\\
 	d_{0} \cdot \mu_{2} & \text{if $s=1$}
 	\end{cases*}
 	$}
\end{align*}
for some $d_{0} \in \mathbb{K}$, where the arrow $A \mapsto B$ means to replace $A$ by $B$.\par
Now if $s \geqslant 2$ and assuming we have found $S_{i_{0}}=\{f_{i},g_{i}:i \geqslant i_{0}+1\}$ for some $1 \leqslant i_{0} \leqslant s-1$ satisfying (4.4) except that if there exist $i$ such that $i \geqslant i_{0}+2$ then $f_{i} \cdot g_{i-1}=c_{i} \cdot \mu_{1}$ for some nonzero $c_{i} \in \mathbb{K}$ . Then we furtherly choose $f_{i_{0}} \in V^{(i_{0}-1)} \backslash V^{(i_{0})}$,  $g_{i_{0}} \in V_{(i_{0})} \backslash V_{(i_{0}-1)}$ and $h_{i_{0}-1} \in V_{(i_{0}-1)} \backslash V_{(i_{0}-2)}$\,  s.t. $f_{i_{0}} \cdot h_{i_{0}-1} =\mu_{1}$, $F(f_{i_{0}},g_{i_{0}})=1$ and $F(f_{i_{0}},f_{i_{0}})=0$. And we normalize the multiplications through the following steps. Firstly:
\begin{align*}
f_{i_{0}}\mapsto f_{i_{0}}-\sum_{i=i_{0}+1}^{s} (F(f_{i_{0}},f_{i}) \cdot g_{i}+F(f_{i_{0}},g_{i}) \cdot f_{i}) - \sum_{k=1}^{p} F(e_{k},f_{i_{0}}) \cdot e_{k} 
\end{align*}
to make $F(\alpha,f_{i_{0}})=0$ for all $\alpha \in S_{i_{0}} \cup \{e_{k}:1 \leqslant k \leqslant p\}$. Then
\begin{align*}
\alpha    &\mapsto \alpha- V_{1}(\alpha,f_{i_{0}}) \cdot h_{i_{0}-1}  &&\text{to make $\alpha \cdot f_{i_{0}}=0$}\\
g_{i_{0}} &\mapsto  g_{i_{0}}-V_{1}(f_{i_{0}},g_{i_{0}}) \cdot h_{i_{0}-1} &&\text{to make $f_{i_{0}} \cdot g_{i_{0}}=b_{0}$}\\
f_{i_{0}} &\mapsto  f_{i_{0}}-(V_{1}(f_{i_{0}},f_{i_{0}})/2) \cdot h_{i_{0}-1} &&\text{to make $f_{i_{0}}^2=
	\begin{cases*}
	0 &\text{if $i_{0} \geqslant 2$}\\
	d_{0} \cdot \mu_{2} & \text{if $i_{0}=1$}
	\end{cases*}
	$}
\end{align*}\par
\noindent for some $d_{0} \in \mathbb{K}$.
Moreover from the discussion at the beginning we have $f_{i_{0}+1} \cdot g_{i_{0}} =c_{i_{0}+1} \cdot \mu_{1}$  with some nonzero $c_{i_{0}+1} \in \mathbb{K}$. And from $\mathfrak{m}^{2} \subseteq \langle \mu_{1},\mu_{2},b_{0}\rangle,b_{0} \cdot \mathfrak{m}=0,V^{(1)} \cdot \mathfrak{m} \subseteq \langle b_{0},\mu_{1} \rangle$ we have $d_{0} =0 $ if and only if $dim(\mathfrak{m}^2)=2$. Finally note that the symmetric matrix $\Lambda=(V_{1}(e_{k},e_{l}))$ under orthogonal transformations on  $\{e_{1},...,e_{p}\}$ transforms as the matrix of a bilinear form. And a such transformation will not affect our normalization on other elements, hence from \cite[Chapter XI \S 3]{matrices}, $\Lambda=(V_{1}(e_{i},e_{j}))$ can be transformed into a canonical symmetric block diagonal matrix (see (\ref{canonical_form}) in Proposition \ref{classify_A}).\par
To finish our normalization it suffices to make $f_{i} \cdot g_{i-1}=\mu_{1}$ and $f_{1}^2=\delta \cdot \mu_{2}$. To do this we firstly replace $f_{i}$ by $x_{i} \cdot f_{i}$ and replace $g_{i}$ by $y_{i} \cdot g_{i}$. Then the condition $(f_{i} \cdot g_{i-1}=\mu_{1},f_{1}^2=\delta \cdot \mu_{2},f_{i} \cdot g_{i}=b_{0})$ gives a system of equations for $\{x_{i},y_{j}:1 \leqslant i \leqslant s, 0 \leqslant j \leqslant s\}$:
\begin{equation}
x_{i} \cdot y_{i}=1,\,
x_{i} \cdot y_{i-1}=c_{i}^{-1},\,
x_{1}^2 \cdot d_{0}=y_{0} \cdot \delta
\end{equation}
 for which we have a solution to be calculated inductively:\\
 ($\delta=1$)
 $\begin{cases}
 	x_{i}=y_{i}^{-1}\\
 	y_{i}=y_{i-1} \cdot c_{i} \\
 	y_{0}=(\frac{d_{0}}{c_{1}^{2}})^{\frac{1}{3}}
 \end{cases}$  \,\,\,\,\,\,\,\,\,\, and \,\,\,\,\,\,\,\,\,\,\,\,\, ($\delta=0$)
 $\begin{cases}
 x_{i}=y_{i}^{-1}\\
 y_{i}=y_{i-1} \cdot c_{i} 
 \end{cases}$ \\
 concluding the normalization.\\
(ii) If $V^{(s)}=V_{(s)}$ then the process of normalization will be the same as in (i) except that we do not need to choose $e_{k}$ at the beginning. 
\end{proof}
 Following our normalization we can thus determine the normalized  structure of $(R,W,F)$ in Case 1.
 \begin{proposition}[Classification of $Type$ $ A$]\label{classify_A} $(R,W,F)$ can be transformed into the following:\\
 	$\bullet \,Type$ $A_{1}$\label{A_1}: $Q^{(s)}$ is a projective space (equivalently $V_{(s)}=V^{(s)}$)
 	\begin{equation*}
 	M(F,Type A_{1})=
 	\left(\begin{matrix}
 	0 & 0 & 0      & 0     & 0     & 0      & 0      & 0 \\
 	0 & 0 & 0      & 0     & 0     & 0      & 0      & 0 \\
 	0 & 0 &0       & \dots  & 0     &0    & \dots    & 1 \\
 	0 & 0 & \vdots & \dots  & 0     & \vdots & \reflectbox{$\ddots$} & \vdots \\
 	0 & 0 &     0  & \dots  & 0     & 1      & \dots & 0   \\
 	0 & 0 & 0      & \dots & 1     & 0      & \dots & 0\\
 	0 & 0 & \vdots & \reflectbox{$\ddots$}& \vdots & 0      & \dots & 0\\
 	0 & 0 & 1     & \dots & 0     & 0      & \dots & 0
 	\end{matrix}\right), 
 	\end{equation*}
 	$W=\langle \mu_{1},\mu_{2} \rangle \oplus \langle g_{1},...,g_{s} \rangle \oplus \langle f_{s},...,f_{1} \rangle$. \\
 	$R \cong \mathbb{K}[\mu_{1},\mu_{2},g_{1},...,g_{s},f_{1},...,f_{s}]/(\mu_{1}\cdot W,g_{i}\cdot \mu_{2},f_{l} \cdot \mu_{2},g_{i} \cdot f_{i}-g_{v} \cdot f_{v},g_{l-1} \cdot f_{l}-\mu_{1},f_{1} \cdot \mu_{2}-\mu_{1},f_{1}^2-\delta \cdot \mu_{2},g_{i} \cdot g_{v},f_{h} \cdot g_{h'},f_{l} \cdot f_{i}, 1 \leqslant i,v \leqslant s ,  \, h-h' \not \in \{0,1\} $, $2 \leqslant l \leqslant s$ when $s \geqslant 2 )$ where 
 	\[
 	\delta=
 	\begin{cases*}
 	0 &\text{if $dim(\mathfrak{m}^2)=2$;}\\
 	1 & \text{if $dim(\mathfrak{m}^3)=3$}
 	\end{cases*}
 	\]\\
 	$\bullet \, Type\, A_{2}$\label{A_2}: $Q^{(s)}$ is a hyperquadric (equivalently $V_{(s)} \not = V^{(s)}$).
 	\begin{equation*}
 	M(F,Type A_{2})=
 	\left(\begin{matrix}
0  &0 & 0      & 0     & 0      & 0      & 0     & 0      & 0      & 0     & 0 \\	
0  &0 & 0      & 0     & 0      & 0      & 0     & 0      & 0      & 0     & 0 \\
0  &0 & 0      & \dots & 0      & 0      & \dots & 0      & 0      & \dots & 1 \\
0  &0 & \vdots & \dots & 0      & \vdots & \dots & \vdots & \vdots & \reflectbox{$\ddots$} & \vdots \\
0  &0 & 0      & \dots & 0      & 0      & \dots & 0      & 1      & \dots & 0      \\
0  &0 & 0      & \dots & 0      & 1     & \dots & 0      & 0      & \dots & 0 	 \\
0  &0 & \vdots & \dots & 0      & \vdots & \ddots& \vdots & \vdots & \dots & \vdots\\
0  &0 & 0      & \dots & 0      & 0      & \dots & 1     & 0      & \dots & 0\\
0  &0 & 0      & \dots & 1      & 0      & \dots & 0      & 0      & \dots & 0\\
0  &0 & \vdots & \reflectbox{$\ddots$}& \vdots & \vdots & \dots & \vdots & 0      & \dots & 0\\
0  &0 & 1      & \dots & 0     & 0      & \dots & 0      & 0      & \dots & 0
 	\end{matrix}\right), 
 	\end{equation*}
 	$W=\langle \mu_{1},\mu_{2} \rangle \oplus \langle g_{1},...,g_{s} \rangle \oplus \langle e_{1},...,e_{p}\rangle  \oplus \langle f_{s},...,f_{1} \rangle$.\\
 	$R \cong \mathbb{K}[\mu_{1},\mu_{2},g_{1},...,g_{s},e_{1},...,e_{p},f_{1},...,f_{s}]/(\mu_{1}\cdot W,g_{i}\cdot \mu_{2},e_{i'} \cdot \mu_{2},f_{l} \cdot \mu_{2},g_{i} \cdot f_{i}-e_{i'}^2+\lambda_{i'i'}\mu_{1},e_{i'} \cdot e_{i''}-\lambda_{i'i''}\mu_{1},g_{l-1} \cdot f_{l}-\mu_{1},f_{1} \cdot \mu_{2}-\mu_{1},f_{1}^2-\delta \cdot \mu_{2},g_{i} \cdot g_{v},f_{l} \cdot f_{i},f_{h} \cdot g_{h'},e_{i'} \cdot f_{i},e_{i'} \cdot g_{i}\,, 1 \leqslant i,v \leqslant s ,1 \leqslant i' \not = i'' \leqslant p,  h-h' \not \in \{0,1\}, 2 \leqslant l \leqslant s$ when $s \geqslant 2)$	where $\delta$ is the same as in $Type \, A_{1}$ and $\Lambda= (\lambda_{i'i''})$ is of the standard form, i.e., a symmetric block diagonal $t \times t$-matrix such that each block $\Lambda_{k}$ is 
 	\begin{equation}\label{canonical_form}
 	\lambda_{k}  \left(
 	\begin{matrix}
 	
 	1     & 0         &          & 0 \\
 	0     & \ddots  &\ddots  &   \\
 	& \ddots  &\ddots  & 0 \\
 	0     &           &0         & 1
 	\end{matrix}
 	\right) +
 	\frac{1}{2}\left(
 	\begin{matrix}
 	
 	0     & 1         &          & 0 \\
 	1     & \ddots  &\ddots  &   \\
 	& \ddots  &\ddots  & 1  \\
 	0     &           &1         & 0
 	\end{matrix}
 	\right) +
 	\frac{i}{2} \left(
 	\begin{matrix}
 	
 	0     &           &1         & 0 \\
 	& \ddots  &\ddots  & -1 \\
 	1     & \ddots  &\ddots  & 0 \\
 	0     & -1        &          & 0
 	\end{matrix}
 	\right)
 	\end{equation}
 	with some $\lambda_{k} \in \mathbb{K}$.
 \end{proposition} 
\case $x=B$. In this case the sequence becomes
\begin{align*}
\langle \mu_{1}\rangle =V_{(-1)} \subseteq Ker(F)=V_{(0)} \subseteq...\subseteq V_{(s)}\subseteq V^{(s)}  \subseteq ... \subseteq V^{(0)}=W,
\end{align*}
with $V^{(k)} \cdot V_{(k-1)}=0$ for $k \geqslant 0$ and $V_{(s)}^2 \not =0$. We have the following normalization.
\begin{lemma}\label{norm_B}
(i) If $V^{(s)} \not = V_{(s)}$ and $s \geqslant 1$ then there exist $f_{i} \in V^{(i-1)} \backslash V^{(i)}$, $g_{i} \in V_{i} \backslash V_{(i-1)}$ for $1\leqslant i \leqslant s $, $g_{0} \doteq \mu_{2}$ and $\{e_{k}: 1 \leqslant k \leqslant p\} \subseteq V^{(s)} \backslash V_{(s)}$ such that
 \begin{equation}
 \begin{split}
 V^{(s)}&=V_{(s)}\oplus \langle e_{1},...,e_{p} \rangle, \\
 e_{k} \cdot e_{l} &=\delta_{k,l} \cdot b_{0}+V_{1}(e_{k},e_{l}) \cdot \mu_{1}, \\
 \end{split}
 \end{equation}
 and
 \begin{equation}
 \begin{split}
 e_{k} \cdot f_{i}&=e_{k} \cdot g_{i}=f_{i} \cdot f_{j}=f_{v} \cdot g_{v'}=0, \\
 f_{i} \cdot g_{i}&=b_{0}, \,	f_{i} \cdot g_{i-1}=\mu_{1}, \, f_{1}^2=\delta \cdot \mu_{2}, \, g_{s}^2=\mu_{1},
 \end{split}
 \end{equation} 
 	for $1 \leqslant i \leqslant s$, $1 \leqslant k, l \leqslant p $, $ v-v' \not \in \{0,1\}$, $ 2\leqslant j \leqslant s$ when $s \geqslant 2$,
 \[
 \delta=
 \begin{cases*}
 0 &\text{if $dim(\mathfrak{m}^2)=2$;}\\
 1 & \text{if $dim(\mathfrak{m}^3)=3$}
 \end{cases*}
 \] 
 and  the matrix $\Lambda=(V_{1}(e_{k},e_{l}))$ is of the canonical form (4.6).\par
 (ii) If $V^{(s)}=V_{(s)}$ and $s \geqslant 1$ then there exists $f_{i} \in  V^{(i-1)} \backslash V^{(i)}$, $g_{i} \in V_{(i)} \backslash V_{(i-1)}$ for $1 \leqslant i \leqslant s$, $g_{0} =\mu_{2}$ such that
 \begin{align*}
 f_{i} \cdot g_{i}&=b_{0}, \,	f_{i} \cdot g_{i-1}=\mu_{1}, \,f_{i} \cdot f_{j}=f_{v} \cdot g_{v'}=0, \, f_{1}^2=\delta \cdot \mu_{2}, \, g_{s}^2=\mu_{1},
 \end{align*} 
 for $1 \leqslant i  \leqslant s$, $v-v' \not \in \{0,1\}$, $ 2\leqslant j \leqslant s$ when $s \geqslant 2$, and $\delta$ is the same as in (i).\par
 (iii) If $s=0$ then there exists a basis of $Ker(F)=\langle \mu_{1},\mu_{2} \rangle$ and $\{e_{k}:1 \leqslant k \leqslant p\}$ such that $\mu_{2}^2=\mu_{1},e_{k} \cdot \mu_{2}=0$ and 
 \begin{equation*}
 \begin{split}
W&=Ker(F)\oplus \langle e_{1},...,e_{p} \rangle \\
e_{k} \cdot e_{l} &=\delta_{k,l} \cdot b_{0}+V_{1}(e_{k},e_{l}) \cdot \mu_{1}, \notag \\
\end{split}
 \end{equation*}
and the matrix $\Lambda=(V_{1}(e_{k},e_{l}))$ is of the canonical form (4.6).
\end{lemma}
\begin{proof}
	
	(i) As in Case 1.(i) we can first choose $e_{k}$ satisfying (4.7). Also we can choose $f_{s} \in V^{(s-1)}\backslash V^{(s)}$, $g_{s} \in V_{(s)} \backslash V_{(s-1)}$ and  $h_{s-1} \in V_{(s-1)} \backslash V_{(s-2)}$ s.t. $f_{s} \cdot h_{s-1} =\mu_{1}$, $F(f_{s},g_{s})=1$ and $F(f_{s},f_{s})=0$. Then as $V_{(s)} \cdot V_{(s)} \not =0$, $V_{(s-1)} \cdot V_{(s)} \subseteq V_{(s-1)} \cdot V^{(s)}=0$ and $codim(V_{(s-1)},V_{(s)})=1$ we conclude that $g_{s}^2$ is a nonzero element in $\langle \mu_{1} \rangle$ and hence we can assume $g_{s}^2=\mu_{1}$. Now we normalize the multiplications between $f_{s},g_{s},e_{k}$ through the following steps:
	\begin{align*}
	e_{k} & \mapsto e_{k}-V_{1}(g_{s},e_{k}) \cdot g_{s} &&\text{to make $g_{s} \cdot e_{k}=0$}\\
	f_{s} &\mapsto  f_{s}- \sum_{k=1}^{p} F(e_{k},f_{s}) \cdot e_{k}  &&\text{to make $F(f_{s},e_{k})=0$}\\
	e_{k} & \mapsto e_{k}-V_{1}(f_{s},e_{k}) \cdot h_{s-1} && \text {to make $f_{s} \cdot e_{k}=0$}\\
	g_{s} & \mapsto g_{s}-V_{1}(g_{s},f_{s}) \cdot h_{s-1} && \text{to make $f_{s} \cdot g_{s}=b_{0}$}\\
	f_{s} &\mapsto  f_{s}-(V_{1}(f_{s},f_{s})/2) \cdot h_{s-1} &&\text{to make $f_{s}^2=
		\begin{cases*}
		0 &\text{if $s \geqslant 2$}\\
		d_{0} \cdot \mu_{2} & \text{if $s=1$}
		\end{cases*}
		$}
	\end{align*} for some $d_{0} \in \mathbb{K}$.\par
	After this note that we can still use previous inductive operations in Case 1 to find $f_{i},g_{i}$ for $i \leqslant s-1$, namely we can find suitable $f_{i},g_{i}$ satisfying (4.8) except that we have $f_{i} \cdot g_{i-1}=c_{i} \cdot \mu_{1}$ and $f_{1}^2=d_{0} \cdot \mu_{2}$ for some nonzero $c_{i} \in \mathbb{K}$ and $d_{0}=0$ if and only if $dim(\mathfrak{m}^2)=2$.  Also for the same reason as in Case 1 we can assume $\Lambda=(V_{1}(e_{k},e_{l}))$ is of the canonical form.\par
    Now to finish our  normalization we replace $f_{i}$ by $x_{i} \cdot f_{i}$, replace $g_{i}$ by $y_{i} \cdot g_{i}$ and replace $\mu_{1}$ by $z_{0} \cdot \mu_{1}$. Then it suffices to satisfy the condition $(f_{i} \cdot g_{i-1}=\mu_{1}, f_{1}^2=\delta \cdot \mu_{2}, f_{i} \cdot g_{i}=b_{0},g_{s}^2=\mu_{1} )$, which gives a system of equations for $\{x_{i},y_{j},z_{0} \in \mathbb{K}:1 \leqslant i \leqslant s, 0 \leqslant j \leqslant s \}$:
	\begin{equation*}
	x_{i} \cdot y_{i}=1, \,x_{i} \cdot y_{i-1}=z_{0} \cdot c_{i}^{-1},\, x_{1}^2 \cdot d_{0}=y_{0} \cdot \delta,\,	y_{s}^2=z_{0}
	\end{equation*}
	for which we have a solution (where $c \doteq \prod_{i=1 }^s c_{i}$ ) to be calculated inductively:  \\
	$(\delta=1)$ 
	$\begin{cases}
	x_{i}=y_{i}^{-1}\\
	y_{i}=y_{i-1} \cdot c_{i} \cdot z_{0}^{-1}\\
	z_{0}=(\frac{c^{3} \cdot d_{0}}{c_{1}^{2}})^{\frac{2}{6s-1}},
	y_{0}=(\frac{d_{0}\cdot z_{0}^{2}}{c_{1}^{2}})^{\frac{1}{3}}
	\end{cases}$
	 \,\,\,\,\,\,\, and \,\,\,\,\,\, 
	$(\delta=0)$ 
	$\begin{cases}
	x_{i}=y_{i}^{-1}\\
	y_{i}=y_{i-1} \cdot c_{i} \\
	y_{0}=c^{-1},
	z_{0}=1
	\end{cases}$\\
	concluding our normalization.\par
	(ii) If $V^{(s)}=V_{(s)}$ then the process will be the same as (i) except that we do not need to choose $e_{k}$ at the beginning.\par
	(iii) Note that $s=0$ is equivalent to $Ker(F) \cdot Ker(F) \not =0$, which is just the Case 2 in the proof of Proposition \ref{mainprosec2}. Hence the assertion follows from Lemma 3.9.
\end{proof}
We can now determine the algebraic structure of $(R,W,F)$ in Case 2.
\begin{proposition}[Classification of $Type$ $B$]\label{classify_B}
	$(R,W,F)$ can be transformed into the following:\\
	$\bullet \,Type$ $B_{0}$\label{B_0} : $s=0$ ($Ker(F) \cdot Ker(F)\not=0$)
	\begin{equation*}
	M(F,Type B_{0 })=
	\left(\begin{matrix}
	0  &0      & 0      & 0     & 0  \\      
	0  &0       & 0      & 0     & 0   \\   
	0  &0       & 1     & \dots & 0   \\  
	0  &0       & \vdots & \ddots& \vdots \\
	0  &0       & 0      & \dots & 1     
	\end{matrix}\right) , \,\, W=\langle \mu_{1},\mu_{2} \rangle \oplus \langle e_{1},...,e_{p} \rangle 
	\end{equation*}
	$R \cong \mathbb{K}[\mu_{1},\mu_{2},e_{1},...,e_{p}]/(\mu_{1} \cdot W,\mu_{2}^{2}-\mu_{1},e_{i} \cdot \mu_{2},e_{i} \cdot e_{j}-\lambda_{i,j} \cdot \mu_{1},e_{i}^{2}-e_{j}^{2}-(\lambda_{i,i}-\lambda_{j,j}) \cdot \mu_{1}, \, 1\leqslant i \not =j \leqslant p)$ where $\Lambda=(\lambda_{i,j})$ is of the canonical form (\ref{canonical_form}).\\\\
	$\bullet \,Type$ $B_{1}$\label{B_1}: $s\geqslant 1$ and $Q^{(s)}$ is a projective space (equivalently $V_{(s)}=V^{(s)}$).
	\begin{equation*}
	M(F,Type B_{1})=M(F,Type A_{1}),
	\end{equation*}
	$W=\langle \mu_{1},\mu_{2} \rangle \oplus \langle g_{1},...,g_{s} \rangle \oplus \langle f_{s},...,f_{1} \rangle$,\\
	$ R \cong \mathbb{K}[\mu_{1},\mu_{2},g_{1},..,g_{s},f_{1},..,f_{s}]/(\mu_{1}\cdot W,g_{i}\cdot \mu_{2},f_{l} \cdot \mu_{2},g_{s}^2-\mu_{1},g_{i} \cdot f_{i}-g_{v} \cdot f_{v},g_{l-1} \cdot f_{l}-\mu_{1},f_{1} \cdot \mu_{2}-\mu_{1},f_{1}^2-\delta \cdot \mu_{2},f_{h} \cdot g_{h'},f_{i} \cdot f_{l},g_{i} \cdot g_{l' }, \, 1 \leqslant i \not=v  \leqslant s$, $h-h' \not \in \{0,1\} $, $ 2 \leqslant l \leqslant s$, $1 \leqslant l' \leqslant s-1$ when $s \geqslant 2)$ where $\delta$ is the same as in Type $A_{1}$.\\\\
	$\bullet \,Type\, B_{2}$\label{B_2}: $s \geqslant 1$ and $Q^{(s)}$ is a hyperquadric (equiuvalently $V_{(s)} \not = V^{(s)}$).
	\begin{equation*}
	M(F,Type B_{2})=M(F,Type A_{2}),
	\end{equation*}
	$W=\langle \mu_{1},\mu_{2} \rangle \oplus \langle g_{1},...,g_{s} \rangle \oplus \langle e_{1},...,e_{p}\rangle  \oplus \langle f_{s},...,f_{1} \rangle$,\\
	$ R \cong \mathbb{K}[\mu_{1},\mu_{2},g_{1},...,g_{s},e_{1},...,e_{p},f_{1},...,f_{s}]/(\mu_{1}\cdot W,g_{i}\cdot \mu_{2},e_{i'} \cdot \mu_{2}, f_{l} \cdot \mu_{2}, g_{s}^2-\mu_{1},g_{i} \cdot f_{i}-e_{i'}^2+\lambda_{i'i'}\mu_{1},g_{l-1} \cdot f_{l}-\mu_{1},e_{i'} \cdot e_{i''}-\lambda_{i'i''}\mu_{1},e_{i'} \cdot f_{i},e_{i'} \cdot g_{i}, f_{1} \cdot \mu_{2}-\mu_{1},f_{1}^2-\delta \cdot \mu_{2},g_{i} \cdot g_{l'},f_{i} \cdot f_{l},f_{h} \cdot g_{h'}, \,  1 \leqslant i  \leqslant s ,1 \leqslant i' \not = i'' \leqslant p, h-h' \not \in \{0,1\} $, $2 \leqslant l \leqslant s$, $1 \leqslant l' \leqslant s-1$ when $s \geqslant 2)$\\
	 where $\Lambda= (\lambda_{i'i''})$ is of the canonical form (\ref{canonical_form}) and $\delta$ is the same as in $Type \, A_{1}$.
\end{proposition}
\case $x=C$. In this case the algebraic sequence becomes
\begin{align*}
\langle \mu_{1} \rangle =V_{(-1)} \subseteq Ker(F)=V_{(0)} \subseteq...\subseteq V_{(s)}=V_{(s+1)} \subseteq V^{(s+1)} \subseteq V^{(s)}  \subseteq ... \subseteq V^{(0)}=W,
\end{align*}
with $V^{(k)} \cdot V_{(k-1)}=0$ if $k \geqslant 0 $ and $V_{(s)} \cdot V_{(s)}=0$.  We have the following normalization.
\begin{lemma}\label{norm_C}
	 (i) If $V_{(s+1)} \not =V^{(s+1)}$ and $s \geqslant 1$ then there exist $f_{i} \in V^{(i-1)} \backslash V^{(i)}$, $g_{j} \in V_{(j)}\backslash V_{(j-1)}$ for $1 \leqslant i \leqslant s+1, 1\leqslant j \leqslant s$, $g_{0} \doteq \mu_{2}$ and $\{e_{k}:1 \leqslant k \leqslant p\} \subseteq V^{(s+1)} \backslash V_{(s+1)}$ such that 
	 \begin{equation}
	 \begin{split}
	 V^{(s+1)}&=V_{(s+1)}\oplus \langle e_{1},...,e_{p} \rangle \\
	 e_{k} \cdot e_{l} &=\delta_{k,l} \cdot b_{0}+V_{1}(e_{k},e_{l}) \cdot \mu_{1}  \\
	 \end{split}
	 \end{equation}
	 and
	 \begin{equation}
	\begin{split}
	 e_{k} \cdot f_{i}&=e_{k} \cdot g_{j}=f_{j} \cdot f_{j'}=f_{v} \cdot g_{v'}=0 \\
	 f_{j} \cdot g_{j}&=f_{s+1}^2=b_{0}, \,	f_{i} \cdot g_{i-1}=\mu_{1}, \, f_{1}^2=\delta \cdot \mu_{2},
	 \end{split}
	 \end{equation}
	 for $1 \leqslant i \leqslant s+1$, $1 \leqslant k,l \leqslant p $, $v-v' \not \in \{0,1\}$, $ 1\leqslant j \leqslant s$, $ 2\leqslant j' \leqslant s+1$,
	 \[
	 \delta=
	 \begin{cases*}
	 0 &\text{if $dim(\mathfrak{m}^2)=2$;}\\
	 1 & \text{if $dim(\mathfrak{m}^3)=3$}
	 \end{cases*}
	 \] 
	 and  the matrix $\Lambda=(V_{1}(e_{k},e_{l}))$ is of the canonical form (4.6).\\
	 (ii) If $V^{(s+1)}=V_{(s+1)}$ and $s \geqslant 1$ then there exist $f_{i} \in  V^{(i-1)} \backslash V^{(i)}$, $g_{j} \in V_{(j)} \backslash V_{(j-1)}$ for $1 \leqslant i \leqslant s+1,1 \leqslant j \leqslant s$, $g_{0} =\mu_{2}$ such that
	 \begin{align*}
	 f_{j} \cdot f_{j'}=f_{v} \cdot g_{v'}=0, \,
	 f_{j} \cdot g_{j}=f_{s+1}^2=b_{0}, \,	f_{i} \cdot g_{i-1}=\mu_{1}, \, f_{1}^2=\delta \cdot \mu_{2},\notag 
	 \end{align*} 
	 for $1 \leqslant i  \leqslant s+1$, $v-v' \not \in \{0,1\}$, $ 1\leqslant j \leqslant s$, $ 2\leqslant j' \leqslant s+1$, and $\delta$ is the same as in (i).\\
	 (iii) If $s=0$ then there exist $\{e_{k}:1 \leqslant k \leqslant p+1\}$ such that $e_{p+1} \in W \backslash V^{(1)}$ and: 
	 \begin{align*}
	 V^{(1)}&=\langle \mu_{1},\mu_{2} \rangle \oplus \langle e_{1},...,e_{p} \rangle,\\
	 e_{k} \cdot e_{l} &=\delta_{k,l} \cdot b_{0}+V_{1}(e_{k},e_{l}) \cdot \mu_{1},\\
	 e_{p+1} \cdot \mu_{2}&=\mu_{1}, e_{p+1} \cdot e_{k}=0 ,e_{p+1}^{2}=b_{0}+\delta \cdot \mu_{2},	 
	 \end{align*}
	 for $1 \leqslant k,l \leqslant p$ and $\delta$ is the same as in (i).
\end{lemma}
\begin{proof}
(i) Firstly we can choose $e_{k}$ satisfying (4.9) and from $V_{(s)}=Ker(F_{|_{V^{(s)}}})$ we can choose $f_{s+1} \in V^{(s)} \backslash V^{(s+1)}$ s.t. $F(e_{k},f_{s+1})=0$ and $F(f_{s+1},f_{s+1})=1$. Furtherly we can choose $f_{s} \in V^{(s-1)} \backslash V^{(s)}$, $g_{s} \in V_{(s)} \backslash V_{(s-1)}$ and $h_{s-1} \in V_{(s-1)} \backslash V_{(s-2)}$ s.t. $f_{s} \cdot h_{s-1}=\mu_{1}$, $F(f_{s},g_{s})=1$ and $F(f_{s},f_{s})=0$. Moreover we have $f_{s+1} \cdot g_{s}=c_{s+1} \cdot \mu_{1}$ for some nonzero $c_{s+1} \in \mathbb{K}$. Now we normalize the multiplications between $e_{k},f_{s},g_{s},f_{s+1}$ through the following steps:
\begin{align*}
e_{k} & \mapsto e_{k}-c_{s+1}^{-1} \cdot V_{1}(f_{s+1},e_{k}) \cdot g_{s} &&\text{to make $f_{s+1} \cdot e_{k}=0$}\\
f_{s+1} &\mapsto  f_{s+1}-c_{s+1}^{-1}\cdot (V_{1}(f_{s+1},f_{s+1})/2) \cdot g_{s} &&\text{to make $f_{s+1}^2=b_{0}$}\\
f_{s} & \mapsto f_{s}-\sum_{k=1}^{p} F(e_{k},f_{s}) \cdot e_{k}-F(f_{s},f_{s+1}) \cdot f_{s+1} &&\text{to make $F(f_{s},e_{k})=F(f_{s},f_{s+1})=0$}
\end{align*}
and for any $\alpha \in \{e_{k},f_{s+1}:1\leqslant k \leqslant p\}$
\begin{align*}
\alpha&\mapsto \alpha -V_{1}(\alpha,f_{s}) \cdot h_{s-1} &&\text{to make $f_{s} \cdot \alpha=0$}\\
g_{s} &\mapsto g_{s}-V_{1}(f_{s},g_{s}) \cdot h_{s-1} &&\text{to make $g_{s} \cdot f_{s}=b_{0}$} \\
f_{s} &\mapsto  f_{s}-(V_{1}(f_{s},f_{s})/2) \cdot h_{s-1} &&\text{to make $f_{s}^2=
	\begin{cases*}
	0 &\text{if $s \geqslant 2$}\\
	d_{0} \cdot \mu_{2} & \text{if $s=1$}
	\end{cases*}
	$}
\end{align*}
After this as in Case 1 and 2, we can still inductively find suitable $f_{i},g_{i}$ satisfying (4.10) except that we have $f_{i} \cdot g_{i-1}=c_{i}\cdot \mu_{1}$ and $f_{1}^{2}=d_{0} \cdot \mu_{2}$ for some nonzero $c_{i} \in \mathbb{K}$ and $d_{0}=0$ if and only if $dim(\mathfrak{m}^2)=2$. Also we can assume $\Lambda=(V_{1}(e_{k},e_{l}))$ is of the canonical form (4.6).\par
Now to finish our normalization we again replace $f_{i}$ by $x_{i} \cdot f_{i}$, replace $g_{j}$ by $y_{j} \cdot g_{j}$ and replace $\mu_{1}$ by $z_{0} \cdot \mu_{1}$. Then the condition $(f_{i} \cdot g_{i-1}=\mu_{1},f_{1}^2=\delta \cdot \mu_{2}, f_{i} \cdot g_{i}=b_{0}, f_{s+1}^{2}=b_{0})$ gives a system of equations for $\{x_{i},y_{j},z_{0}, \in \mathbb{K}:1 \leqslant i \leqslant s+1, 0 \leqslant j \leqslant s \}$:
\begin{equation*}
x_{k} \cdot y_{k}=1,  x_{k} \cdot y_{k-1}=z_{0} \cdot c_{k}^{-1} , x_{s+1}^{2}=1, d_{0} \cdot x_{1}^{2}=\delta \cdot y_{0},x_{s+1} \cdot y_{s}=z_{0} \cdot c_{s+1}^{-1}
\end{equation*}
where $1 \leqslant k \leqslant s$, and for which we have a solution:\\
$(\delta=1)$ 
$\begin{cases}
x_{k}=y_{k}^{-1}\\
y_{k}=y_{k-1} \cdot c_{k} \cdot z_{0}^{-1}\\
z_{0}=(c^3 \cdot c_{1}^{-2} \cdot d_{0})^{\frac{1}{3s+4}}\\
y_{0}=(\frac{z_{0}^{2} \cdot d_{0}}{c_{1}^{2}})^{\frac{1}{3}}\\
x_{s+1}=1
\end{cases}$
\,\,\,\,\,\,\, and \,\,\,\,\,\, 
$(\delta=0)$ 
$\begin{cases}
x_{k}=y_{k}^{-1}\\
y_{k}=y_{k-1} \cdot c_{k} \cdot z_{0}^{-1} \\
z_{0}=1, 
y_{0}=c^{-1},
x_{s+1}=1
\end{cases}$\\
where $c \doteq \prod_{i=0}^{s+1} c_{i}$, concluding the normalization.\par
(ii) If $V_{(s+1)}=V^{(s+1)}$ then similarly the process will be the same as (i) except that we do not need to choose $e_{k}$ at the beginning.\par
(iii) If $s=0$ then the assertion follows from Lemma 3.6.
\end{proof}
\begin{proposition}[Classification of $Type$ $C$]\label{classify_C} $(R,W,F)$ can be transformed into the following:\\
$\bullet \, Type$ $C_{0}$\label{C_0}: $s=0$ (equivalently $V_{(1)}=Ker(F)$)\\	
\begin{equation*}
M(F,Type C_{0})=
\left(\begin{matrix}
0  &0      & 0      & 0     & 0  \\      
0  &0       & 0      & 0     & 0   \\   
0  &0       & 1     & \dots & 0   \\  
0  &0       & \vdots & \ddots& \vdots \\
0  &0       & 0      & \dots & 1     
\end{matrix}\right) , \,\, W=\langle \mu_{1},\mu_{2} \rangle \oplus \langle e_{1},...,e_{p},e_{p+1} \rangle 
\end{equation*}
$R \cong \mathbb{K}[\mu_{1},\mu_{2},e_{1},...,e_{p},e_{p+1}]/(\mu_{1} \cdot W,\mu_{2} \cdot e_{i},e_{i} \cdot e_{j}-\lambda_{i,j} \cdot \mu_{1},e_{i}^{2}-e_{j}^{2}-(\lambda_{i,i}-\lambda_{j,j}) \cdot \mu_{1},e_{p+1} \cdot e_{i},e_{p+1}^{2}-\delta \cdot \mu_{2}-e_{i}^{2}+\lambda_{i,i}\cdot \mu_{1}, \, 1\leqslant i \not = j \leqslant p)$\\
 where $\Lambda= (\lambda_{i,i'})$ is of the canonical form (\ref{canonical_form}) and $\delta$ is the same as in Type $A_{1}$.\\\\
$\bullet \, Type$ $C_{1}$\label{C_1}: $Q^{(s+1)}$ is a projective space (equivalently $V_{(s+1)}=V^{(s+1)}$)
\begin{equation*}
M(F,Type C_{1})=
\left(\begin{matrix}
0 & 0 & 0      & 0     & 0   &0    & 0      & 0      & 0 \\
0 & 0 & 0      & 0     & 0   &0    & 0      & 0      & 0 \\
0 & 0 &0       & \dots  & 0  &0    &0     & \dots  & 1 \\
0 & 0 & \vdots & \dots  & 0  &0    & \vdots & \reflectbox{$\ddots$} & \vdots \\
0 & 0 & 0      & 0      & 0  &0 &1       &0       &0\\
0 & 0 &     0  & \dots  & 0  &1    & 0      & \dots & 0   \\
0 & 0 & 0      & \dots & 1  &0    & 0      & \dots & 0\\
0 & 0 & \vdots & \reflectbox{$\ddots$}& \vdots &0  & 0      & \dots & 0\\
0 & 0 & 1      & \dots & 0   &0   & 0      & \dots & 0
\end{matrix}\right) 
\end{equation*}
$W=\langle \mu_{1},\mu_{2} \rangle \oplus \langle g_{1},...,g_{s} \rangle \oplus \langle f_{s+1},f_{s},...,f_{1} \rangle$\\
$R \cong \mathbb{K}[\mu_{1},\mu_{2},g_{1},...,g_{s},f_{1},...,f_{s},f_{s+1}]/(\mu_{1}\cdot W,g_{i}\cdot \mu_{2},f_{l} \cdot \mu_{2},g_{i} \cdot g_{v},g_{i} \cdot f_{i}-g_{v} \cdot f_{v},g_{l-1} \cdot f_{l}-\mu_{1},f_{l} \cdot f_{l'},f_{h} \cdot g_{h'},f_{1} \cdot \mu_{2}-\mu_{1},f_{1}^2-\delta \cdot \mu_{2},f_{s+1}^2-f_{i} \cdot g_{i}, \, 1 \leqslant i,v \leqslant s  ,2 \leqslant l \leqslant s+1, 1 \leqslant l' \leqslant s, h-h' \not \in \{0,1\} )$  where $\delta$ is the same as in $Type \, C_{0}$.\\\\
$\bullet \,Type\, C_{2}$\label{C_2}: $Q^{(s+1)}$ is a hyperquadric (equivalently $V_{(s+1)} \not = V^{(s+1)}$).
\begin{equation*}
M(F,Type C_{2})=
\left(\begin{matrix}
0  &0 & 0      & 0     & 0      & 0      & 0     & 0       &0     & 0      & 0     & 0 \\	
0  &0 & 0      & 0     & 0      & 0      & 0     & 0       &0     & 0      & 0     & 0 \\
0  &0 & 0      & \dots & 0      & 0      & \dots & 0       &0     & 0     & \dots & 1 \\
\vdots &\vdots & \vdots & \dots & \vdots & \vdots & \dots & \vdots  &\vdots & \vdots &\reflectbox{$\ddots$}& \vdots \\
0  &0 & 0      & \dots & 0      & 0      & \dots & 0       &0     & 1      & \dots & 0      \\
0 &0 & 0      & \dots & 0      & 1     & \dots & 0       &0     & 0      & \dots & 0 	 \\
\vdots  &\vdots & \vdots & \dots & \vdots & \vdots & \ddots& \vdots  &\vdots& \vdots & \dots & \vdots\\
0  &0 & 0      & \dots & 0      & 0      & \dots & 1      &0     & 0      & \dots & 0\\
0  &0 & 0      & 0     & 0      & 0      & 0     &0        &1    &0       &0      &0 \\
0  &0 & 0     & \dots & 1      & 0      & \dots & 0       &0     & 0      & \dots & 0\\
\vdots  &\vdots & \vdots &\reflectbox{$\ddots$}& \vdots & \vdots & \dots & \vdots  &\vdots& \vdots & \dots & \vdots\\
0  &0 & 1      & \dots & 0     & 0      & \dots & 0       &0     & 0      & \dots & 0
\end{matrix}\right) 
\end{equation*}
$W=\langle \mu_{1},\mu_{2} \rangle \oplus \langle g_{1},...,g_{s} \rangle \oplus \langle e_{1},...,e_{p}\rangle  \oplus \langle f_{s+1},f_{s},...,f_{1} \rangle$\\
$R \cong \mathbb{K}[\mu_{1},\mu_{2},g_{1},...,g_{s},e_{1},...,e_{p},f_{1},...,f_{s},f_{s+1}]/(\mu_{1}\cdot W,g_{i}\cdot \mu_{2},e_{i'} \cdot \mu_{2},f_{l} \cdot \mu_{2},g_{i} \cdot g_{v},f_{h} \cdot g_{h'},f_{l} \cdot f_{l'},g_{i} \cdot f_{i}-e_{i'}^2+\lambda_{i'i'}\mu_{1},g_{l-1} \cdot f_{l}-\mu_{1},f_{1} \cdot \mu_{2}-\mu_{1},e_{i'} \cdot e_{i''}-\lambda_{i'i''}\mu_{1},e_{i'} \cdot f_{l'},e_{i'} \cdot f_{s+1},f_{1}^2-\delta \cdot \mu_{2},f_{s+1}^2-f_{i}\cdot g_{i} , \, 1 \leqslant i,v \leqslant s , 2 \leqslant l \leqslant s+1,1 \leqslant i' \not = i'' \leqslant p,1 \leqslant l' \leqslant s, h-h' \not \in \{0,1\})$\\  
where $\Lambda= (\lambda_{i'i''})$ is of the canonical form (\ref{canonical_form}) and $\delta$ is the same as in $Type \, C_{0}$.
\end{proposition} 
\end{mycases}
We now give a characterization of $dim(\mathfrak{m}^{2})$.
\begin{proposition}
	 Given an additive action on a hyperquadric $Q$ of corank two with unfixed singularities and $dim(Q) \geqslant 5$,  we represent it by $(R,W,F)$ with $\mathfrak{m}$ the maximal ideal, then $dim(\mathfrak{m}^{2})=l(\mathbb{G}_{a}^{n},Q)$.	 
\end{proposition}
\begin{proof}
 Note that for any $\alpha \in R$, $dim(\mathbb{G}_{a}^{n} \cdot [\alpha])=dim(\mathfrak{g}(\mathbb{G}_{a}^{n})  \cdot \alpha)=dim( \alpha \cdot W )$. Thus the boundary $Q \backslash O=Q \cap \mathbb{P}(\mathfrak{m}) $ as $dim(\mathbb{G}_{a}^{n} \cdot [x])=dim(\mathbb{G}_{a}^{n} \cdot [1_{R}])=dim(O)$ for any invertible element $x$ in $R$. Then for any $x=[\alpha]$ in the boundary, as $\alpha \in \mathfrak{m}$, we have $dim(\mathbb{G}_{a}^{n} \cdot [\alpha])=dim( \alpha \cdot W )) \leqslant dim(\mathfrak{m}^{2})$, concluding that $l(\mathbb{G}_{a}^{n},Q) \leqslant dim(\mathfrak{m}^{2})$. On the other hand, by our normalization results of each type, we can find a suitable element in the boundary whose orbit has dimension $d=dim(\mathfrak{m}^{2})$ as follows.\par
 If the action is not of Type $B_{0}$ or $C_{0}$. By Lemma 4.2, Lemma 4.4 and Lemma 4.6, we have $[f_{1}] \in Q \backslash O$ and $dim(f_{1} \cdot W)=dim(\mathfrak{m}^{2})$.\par 
 If the action is of Type $B_{0}$. By Lemma 4.4, we have $[e_{1}+i \cdot e_{2}+ \mu_{2}] \in Q \backslash O$ and $dim((e_{1}+i \cdot e_{2}+ \mu_{2}) \cdot W)=dim(\mathfrak{m}^{2})$, where $i^{2}=-1$.\par
 If the action is of Type $C_{0}$. By Lemma 4.6 (iii),  we have $[e_{p+1}+i \cdot e_{1}] \in Q \backslash O $ and $dim((e_{p+1}+i \cdot e_{1}) \cdot W)=dim(\mathfrak{m}^{2})$, where $i^{2}=-1$.
\end{proof}
In the following for a normalized structure of each type we call the normalized basis of $W$, i.e., $\{f_{i},g_{j},e_{k},\mu_{1},\mu_{2}\}$ a set of normalized elements and we call the matrix $\Lambda=(\lambda_{i,j})$ (if exists) to be the canonical matrix of the action.
\subsubsection{Uniqueness}
In this section we finish our classification by showing that the normalized structrue is determined by $l(G_{a}^{n},Q)$ and the canonical matrix(if exists) up to certain elementary transformations.\par
 Given two additive actions on hyperquadrics of corank two with unfixed singularities $(\mathbb{G}_{a}^{n},Q)$ and $(\mathbb{G}_{a}^{n},Q')$ for $n \geqslant 5$, we represent them by $(R,W,F)$ and $(R',W',F')$ respectively, and represent their final outputs in the algebraic version of flow chart by $(x,s,V^{(s)},V_{(s)})$ and $(x',s',V'^{(s')},V'_{(s')})$ respectively. Furthermore define $\{(V^{(k)},V_{(k)}): k \leqslant s\}$ and $\{(V'^{(k)},V_{(k)}'): k \leqslant s'\}$ to be the algebraic sturcture sequences of the two actions. Then we have the following.
\begin{theorem}
(i)	If the two actions are equivalent, i.e., there exist
	\begin{equation*}
	\Gamma:R \mapsto R'
	\end{equation*}
	such that $\Gamma$ is a local $\mathbb{K}$-algebra isomorphism and $\Gamma(W)=W'$. Then\par
	 (i.a) $(R,W,F)$ and $(R',W',F')$ are of the same normalized type with $s=s' $  and $l(\mathbb{G}_{a}^{n},Q)=l(\mathbb{G}_{a}^{n},Q')$.\par
	 (i.b) if they are of Type $A_{1}$, $B_{1}$ or $C_{1}$, then they have the same normalized structure.\par
	 (i.c) if they are not of Type $A_{1}$, $B_{1}$ or $C_{1}$, then their canonical matrices $\Lambda$ and $\Lambda'$ differ up to a permutation of blocks, a scalar multiplication, and adding a scalar matrix (which we call elementary transformations).\\
	Conversely\par
	 (ii) if the two actions are of the same type with $l(\mathbb{G}_{a}^{n},Q)=l(\mathbb{G}_{a}^{n},Q'),s=s'$ and  when they are not of Type $A_{1}$, $B_{1}$ or $C_{1}$, suppose that their canonical matrices differ up to above elementary transformations. Then the two actions are equivalent.
\end{theorem}
  We first prove (i.a) and (i.b).
\begin{proof}[Proof of Theorem 4.9 (i.a) and (i.b)]
(i.a) Firstly as $\Gamma$ is an isomorphism, we conclude that $l(\mathbb{G}_{a}^{n},Q)=l(\mathbb{G}_{a}^{n},Q')$ by Proposition 4.8. By $\Gamma(W)=W'$ and Lemma \ref{F_det_by_R_W} we have $F(a,b)=c \cdot F'(\Gamma(a),\Gamma(b))$ for some nonzero $c \in \mathbb{K}$, for any $a,b \in R$. Then from the algebraic version of the flow chart and our definition of $(V^{(k)},V_{(k)})$ for each $k$, we conclude that $s=s'$, $\Gamma(V_{(k)})=V_{(k)}'$ and $\Gamma(V^{(k)})=V'^{(k)}$ for each $k$, implying that the two actions are of the same normalized type shown in Section 4.1.2.  \par 
(i.b) Note that the set of normalized elements of these types does not contain $e_{k}$ hence the structure only depends on $s$ and $l(\mathbb{G}_{a}^{n},Q)$ by our normalization result, concluding the proof.
\end{proof}
To prove (i.c) and (ii),  we separate it into two cases.\par
\begin{mycases}
\case	  If $s \geqslant 1$,  let $\{\mu_{1},\mu_{2},e_{k},g_{i},f_{1},b_{0}:1\leqslant k \leqslant p,1 \leqslant i \leqslant s\}$ and $\{\mu_{1}',\mu_{2}',e_{k}',g_{i}',f_{1}',b_{0}':1\leqslant k \leqslant p,1 \leqslant i \leqslant s\}$ be the associated elements in the normalized structures respectively. Then the isomorphism $\Gamma$ gives :
\begin{align}
\Gamma(b_{0})&=c_{\Gamma}\cdot b_{0}'+f_{W'},\\
\Gamma(f_{1})&=c_{1} \cdot f_{1}'+f_{1,W'},\\
\Gamma(\mu_{v})&=f_{v,1} \cdot \mu_{1}'+f_{v,2} \cdot \mu_{2}',\\
\Gamma(e_{k})&= \sum_{l=1}^{p}a_{k,l} \cdot e_{l}'+\sum_{i=1}^{s}b_{k,i} \cdot g_{i}' +c_{k,1}\cdot \mu_{1}'+c_{k,2} \cdot \mu_{2}' ,
\end{align}
where $f_{W'} \in W'$, $f_{1,W'} \in V'^{(1)}$ and $v \in\{1,2\}$, $1\leqslant  k \leqslant p$,  $c_{\Gamma},c_{1} \not =0 \in \mathbb{K}$. Moreover we define $A=(a_{k,l})$ then we have the following.
\begin{lemma}
	(i) $F'(\Gamma(a),\Gamma(b))=c_{\Gamma} \cdot F(a,b).$ \par
	(ii) $f_{W'}=\lambda^{(1)}\cdot \mu_{1}'+\lambda^{(2)} \cdot \mu_{2}' \in \langle\mu_{1}',\mu_{2}' \rangle$, $\lambda^{(2)}=f_{1,2}=0$ and $f_{1,1},f_{2,2} \not=0$. \par
	(iii) $A' \cdot A=c_{\Gamma} \cdot I_{p} $.
\end{lemma}
\begin{proof}
(i)  Let $a \cdot b= F(a,b) \cdot b_{0}+(a\cdot b)_{|_{W}}$. Then under the notation of Lemma \ref{y0}  we have:
\begin{align*}
F'(\Gamma(a),\Gamma(b))&=y_{0}'(\Gamma(a\cdot b))=F(a,b) \cdot y_{0}'(\Gamma( b_{0}))\\
&=F(a,b) \cdot y_{0}'(c_{\Gamma} \cdot b_{0}'+f_{W'})=c_{\Gamma} \cdot F(a,b).
\end{align*}
(ii) The first assertion follows from $b_{0} \in \mathfrak{m}^2$ and $\Gamma(\mathfrak{m}^2)=(\mathfrak{m}')^2 \subseteq \langle \mu_{1}',\mu_{2}',b_{0}'\rangle$.  For $\lambda^{(2)}$, from $f_{1} \cdot b_{0}=0$ we have:
\begin{align*}
0=\Gamma(b_{0}) \cdot \Gamma(f_{1})=&(c_{\Gamma}\cdot b_{0}'+\lambda^{(1)}\cdot \mu_{1}'+\lambda^{(2)} \cdot \mu_{2}') \cdot (c_{1}\cdot f_{1}'+f_{1,W'})= c_{1} \cdot \lambda^{(2)} \cdot \mu_{1}',
\end{align*}
concluding that $\lambda^{(2)}=0$ as $c_{1}$ is nonzero in $\mathbb{K}$. For $f_{1,2}$, from $f_{1} \cdot \mu_{1}=0$ we have
\begin{align*}
0=\Gamma(\mu_{1}) \cdot \Gamma(f_{0})=&(f_{1,1}\cdot \mu_{1}'+f_{1,2} \cdot \mu_{2}') \cdot (c_{1}\cdot f_{1}'+f_{1,W'})=c_{1} \cdot f_{1,2}\cdot \mu_{1}',
\end{align*} 
concluding that $f_{1,2}=0$,  hence $f_{1,1} \not =0$ and $f_{2,2} \not =0$.\par
(iii) Using (i) and (4.14) we have :
\begin{align*}
\delta_{k,k'} \cdot c_{\Gamma} &=c_{\Gamma} \cdot F(e_{k},e_{k'})=F'(\Gamma(e_{k}),\Gamma(e_{k'}))\\
&= \sum_{l,l'=1}^{p} \delta_{l,l'} \cdot a_{k,l} \cdot a_{k',l'}=\sum_{l=1}^{p} a_{k,l} \cdot a_{k',l},
\end{align*}
concluding that $A'\cdot A=c_{\Gamma}  \cdot I_{p}$.
\end{proof}
Now we are ready to prove Theorem 4.9 (i.c) and (ii) when $s \geqslant 1$.
\begin{proof}[Proof of Theorem 4.9 (i.c),(ii) when $s \geqslant 1$]
(i.c) Computing $\Gamma(e_{k} \cdot e_{k'})=\Gamma(e_{k}) \cdot \Gamma(e_{k'})$:
\begin{align*}
LHS&=\Gamma(\delta_{k,k'} \cdot b_{0}+\lambda_{k,k'} \cdot \mu_{1})=\delta_{k,k'} \cdot c_{\Gamma} \cdot b_{0}'+(\delta_{k,k'} \cdot \lambda^{(1)}+\lambda_{k,k'} \cdot f_{1,1}) \cdot \mu_{1}'.\\
RHS&=(\sum_{l,l'=1}^{p} a_{k,l} \cdot a_{k',l'} \cdot \delta_{l,l'}) \cdot b_{0}'+ (\delta' \cdot b_{k,s} \cdot b_{k',s}+\sum_{l,l'=1}^{p} a_{k,l} \cdot \lambda_{l,l'}' \cdot a_{k',l'}) \cdot \mu_{1}'.
\end{align*}
where $\delta' \not =0$ if $(g_{s}')^{2}=\mu_{1}$, i.e., the action is of Type $B_{2}$. And we claim in this case $b_{k,s}=0$, implying $\delta' \cdot b_{k,s} \cdot b_{k',s'}=0$. This follows from computing $\Gamma(e_{k} \cdot g_{s})=\Gamma(e_{k}) \cdot \Gamma(g_{s})$ from two sides.\par
Now from $LHS=RHS$ combined with $A' \cdot A=c_{\Gamma} \cdot I_{p}$ we have the equation:
\begin{equation*}
 \Lambda'= (\frac{\lambda^{(1)}}{c_{\Gamma}}) \cdot I_{p}+(\frac{f_{1,1}}{c_{\Gamma}^{2}}) \cdot A'\Lambda A
\end{equation*}
hence from \cite[Chapter XI \S 3]{matrices} we conclude that  $\Lambda$ and $\Lambda'$ differ up to the listed elementary transformations.\par
(ii) If the two actions have the same normalized type with $l(\mathbb{G}_{a}^{n},Q)=l(\mathbb{G}_{a}^{n},Q')$, $s=s'$, $\Lambda$ and $\Lambda'$ differ up to elementary transformations, then by our normalization result, to give the isomorphism between actions it suffices to find a new set of normalized elements of $(R,W)$ having the canonical matrix which equals $\Lambda'$. In the following we find the new normalized set $\{\mu_{1}^{(0)},\mu_{2}^{(0)},g_{j}^{(0)},e_{k}^{(0)},f_{i}^{(0)},b_{0}^{(0)}\} $ case by case.\par
1). (up to a permutation of blocks)\par Note that any permutation of blocks can be induced by a permutation of $\{e_{k}:1 \leqslant k \leqslant p\}$. Hence the new set of normalized elements can be defined through a suitable permutation of $e_{k}$ and identity on other elements.  \par
2). (up to adding a scalar matrix) we assume $\Lambda'=\Lambda+h \cdot I_{p}$ for some nonzero $h \in \mathbb{K}$.\par
In this case it suffices to find a new set of normalized elements with $b_{0}^{(0)}=b_{0}-h \cdot \mu_{1}$, $\mu_{1}^{(0)}=\mu_{1}$ and $e_{k}^{(0)}=e_{k}$. To find the set we run our normalization in Section 4.1.2 starting with $\mu_{1},\mu_{2},b_{0}^{(0)}$ and set $e_{k},f_{i},g_{j}$ to be the initial elements we take at each step of the normalization. Then one can easily check that after running the normalization of each type, the new set of normalized elements meets our need.\par
3). (up to a scalar multiplication) we assume $\Lambda'=h \cdot \Lambda$ for some nonzero $h \in \mathbb{K}$.\par
In this case it suffices to find a new set of normalized elements with $b_{0}^{(0)}=c_{\Gamma} \cdot b_{0},e_{k}^{(0)}=\sqrt{c_{\Gamma}} \cdot e_{k},\mu_{1}^{(0)}=f_{1,1}\cdot \mu_{1}$ for some nonzero $c_{\Gamma},f_{1,1} \in \mathbb{K}$ s.t. $c_{\Gamma}=h \cdot f_{1,1}$. To find the elements, we define $f_{i}^{(0)}=x_{i} \cdot f_{i}$, $g_{j}^{(0)}=y_{j} \cdot g_{j}$ and $\mu_{2}^{(0)}=y_{0} \cdot \mu_{2}$. Then the condition ($f_{i}^{(0)} \cdot g_{i-1}^{(0)}=\mu_{1}^{(0)},f_{i}^{(0)} \cdot g_{i}^{(0)}=b_{0}^{(0)}$) and extra conditions in different types shown in Section 4.1.2 gives a system of equations for each type:\\
(Type $A_{2}$)$\begin{cases}
x_{i} \cdot y_{i}=c_{\Gamma}\\
x_{i} \cdot y_{i-1}=f_{1,1}\\
x_{1}^2 =y_{0}
\end{cases}$
(Type $B_{2}$)$\begin{cases}
x_{i} \cdot y_{i}=c_{\Gamma}\\
x_{i} \cdot y_{i-1}=f_{1,1}\\
x_{1}^2 =y_{0} \\
y_{s}^{2}=f_{1,1}
\end{cases}$\\
(Type $C_{2}$)$\begin{cases}
x_{i} \cdot y_{i}=c_{\Gamma}\\
x_{i} \cdot y_{i-1}=f_{1,1}\\
x_{1}^2=y_{0 }\\
x_{s+1}^{2}=c_{\Gamma}
\end{cases}$ with the condtion $c_{\Gamma}=h \cdot f_{1,1}$ for each type.\par
For these equations one can easily check the existence of solutions, which enables us to find the set of normalized elements we need.\par 
\end{proof}
\case If $s=0$, i.e., they are of Type $B_{0}$ or $C_{0}$, then we can use similar method in Case 1 to prove (i.c) and also to prove (ii) when the two canonical matrices differ from a permutation of blocks or adding a scalar matrix. Hence it sufficies to prove (ii) when $\Lambda$ and $\Lambda'$ differ from a scalar multiplication.\par
As is Case 1, it sufficies to find a new set of normalized elements $b_{0}^{(0)}=c_{\Gamma} \cdot b_{0},e_{k}^{(0)}=\sqrt{c_{\Gamma}} \cdot e_{k},\mu_{1}^{(0)}=f_{1,1}\cdot \mu_{1}$ for some nonzero  $c_{\Gamma},f_{1,1} \in \mathbb{K}$ s.t. $c_{\Gamma}=h \cdot f_{1,1}$.\par 
Now if the action is of Type $B_{0}$  we set
\begin{align*}
\mu_{1}^{(0)}=\mu_{1},\mu_{2}^{(0)}=\mu_{2},e_{i}=\sqrt{h} \cdot e_{i}, b_{0}^{(0)}=h \cdot b_{0}.
\end{align*}
Then one can check $\{\mu_{1}^{(0)},\mu_{2}^{(0)},e_{i}^{(0)},b_{0}^{(0)}: 1 \leqslant i \leqslant p\}$ is the normalized set we need.\par
If the action is of Type $C_{0}$ we set
\begin{align*}
\mu_{1}^{(0)}=\frac{\mu_{1}}{h^3},\mu_{2}^{(0)}=\frac{\mu_{2}}{h^2},e_{i}^{(0)}=\frac{e_{i}}{h},e_{p+1}^{(0)}=\frac{e_{p+1}}{h},b_{0}^{(0)}=\frac{b_{0}}{h^2}.
\end{align*}
Then one can check  $\{\mu_{1}^{(0)},\mu_{2}^{(0)},e_{i}^{(0)},e_{p+1}^{(0)},b_{0}^{(0)}:1 \leqslant i \leqslant p\}$ is the normalized set we need, concluding the proof of Theorem 4.9.
\end{mycases}
As an application of our classification, we now prove Theorem \ref{determine_by_degeneration}.
\begin{proof}[Proof of Theorem 1.8]
	(i) Recall in Section 3.1 we have constructed $(R^{(1)},\mathfrak{m}^{(1)})$ or $(R^{(1)},V^{(1)},F^{(1)})$ to be the corresponding local algebra (and invariant linear form on it) of the obtained action $(G^{(1)},Q^{(1)})$ in Theorem 1.6. Hence combined with the algebraic version of the flow chart, for a normalized structure $\{(V^{(k)},V_{(k)}):k \leqslant s\}$ of an additive action $(R,W,F)$, if the final output of the flow chart is $(x,t,G^{(t)},Q^{(t)})$, then the output action $(G^{(t)},Q^{(t)})$ is represented by
	\begin{align*}
	(R^{(t)},\mathfrak{m}^{(t)})&=(V^{(t)} \oplus \langle 1_{R} \rangle,V^{(t)}) &&\text{if $Q^{(t)}$ is a projective space,}  \\
	(R^{(t)},V^{(t)},F^{(t)})&=(V^{(t)} \oplus \langle b_{0} \rangle \oplus \langle 1_{R} \rangle,V^{(t)},F_{|_{V^{(t)}}}) &&\text{if $Q^{(t)}$ is a hyperquadric,}
	\end{align*}
	where $t=s$ when $x=A$ and $t=s+1$ when $x=B$ or $C$. Then (i) follows by Remark \ref{multi} and by checking the multiplications in $V^{(t)}$ in different types as shown in Lemma 4.2,4.4 and 4.6.\par
	(ii) $l(\mathbb{G}_{a}^{n},Q) \leqslant 3$ follows from Proposition 4.8 and our normalization result of each types. $codim(Q^{(k+1)},Q^{(k)})=1$ follows from Proposition 4.1.  
	\par
	(iii) 
	 Now for two actions if they are equivalent induced by $\Gamma:R \mapsto R'$ then from Theorem 4.9 (i.a) they are of the same normalized type and $t=t',l(\mathbb{G}_{a}^{n},Q)=l(\mathbb{G}_{a}^{n},Q')$. Moreover as $\Gamma(V^{(k)})=V'^{(k)}$ we conclude that $\Gamma_{|_{R^{(t)}}}$ induces an isomorphism between the output actions of the two actions, which proves the only if part of Theorem 1.8.\par
	 For the converse, it suffices to check the condition in Theorem 4.9 (ii). If $s=s'=0$ then they are of the same type $x_{0}$. If $s=s' \geqslant 1$, then as the output action is equivalent, $Q^{(t)}$ and $\widetilde{Q}^{(t')}$
	 are either both hyperquadrics or projective spaces, hence they are of the same type $x_{1}$ or $x_{2}$.\par 
	   Now if they are of Type $A_{2},B_{2}$ or Type $C_{2}$, then consider the isomorphism between local algebras induced by the equivalence of the output actions:
	   \begin{align*}
	   \Gamma^{(t)}:(R^{(t)},V^{(t)},V_{(t)}) \mapsto (R'^{(t)},V'^{(t)},V'_{(t)}),
	   \end{align*} using the same method in the proof of Theorem 4.9 (i.c), $\Gamma^{(t)}$ will induce elementary transformations between the canonical matrices of the two actions. Therefore by Theorem 4.9 (ii) we conclude that the two actions are equivalent. \par 
\end{proof}
\subsection{Classification of actions with unfixed singularities (II): $dim(Q) \leqslant 4$ }
In this subsection we consider the case when $dim(Q) \leqslant 4$. Equivalently for a triple $(R,W,F)$ we have $dim(W) \leqslant 4$.\par 
In the folowing we always take a basis of $Ker(F)=\langle \mu_{1},\mu_{2}\rangle $ satisfying Lemma 3.5 (i). We also take $(V^{(1)},V_{(1)})$ defined in Section 3. Then we give the classification case by case.\\
\begin{mycases}
\case $dim(W)=4$ \\
\subcase $Ker(F) \cdot Ker(F) \not =0$. Note that in the proof of Case 2 of Proposition \ref{mainprosec2} we only need to assure the number of $e_{k}$ is at least two, hence this case is just the 4-dimensional version of Type $B_{0}$.\\
\subcase $Ker(F) \cdot Ker(F)=0$ and $Ker(F)=V_{(1)}$. We have the following:
\begin{align*}
Ker(F)=V_{(1)} \subsetneqq V^{(1)} \subsetneqq W,
\end{align*}
with $codim(V_{(1)},V^{(1)})=1$. In ths case, we can choose a $g_{1} \in V^{(1)} \backslash V_{(1)}$ such that $F(g_{1},g_{1})=1$ as $F(g_{1},g_{1}) \not =0$. Then for any $f_{1} \in W \backslash V^{(1)}$, up to relplacing it by $f_{1}-F(f_{1},g_{1})^{-1} \cdot g_{1}$, we can have $f(f_{1},g_{1})=0$. Finally $F(f_{1},f_{1}) \not =0$ as $f_{1} \not \in Ker(F)$, hence we can have $F(f_{1},f_{1})=1$. Then we divide it into two more subcases.\par
$(ii.1)$ $dim(V^{(1)} \cdot W)=3$ then there exist a basis of $Ker(F)=\langle \mu_{1},\mu_{2} \rangle$, $f_{1}\in W \backslash V^{(1)},g_{1} \in V^{(1)} \backslash Ker(F) $ and $b_{0} \in \mathfrak{m} \backslash W$ s.t. 
\begin{align*}
g_{1}^{2}=b_{0},\,g_{1} \cdot f_{1}=\mu_{2},\,f_{1}^{2}=b_{0}+\lambda \cdot \mu_{2},\,f_{1} \cdot \mu_{2}=\mu_{1},
\end{align*}
for some $\lambda \in \mathbb{K}$.\par
 To show this, from $g_{1} \cdot g_{1} \not \in Ker(F)$, $g_{1} \cdot Ker(F)=0$, $f_{1} \cdot \mu_{2} \in \langle  \mu_{1} \rangle$ and our assumption $dim(V^{(1)} \cdot W)=3$ we conclude that $g_{1} \cdot f_{1}=c_{2} \cdot \mu_{2}+c_{1} \cdot \mu_{1}$ for some nonzero $c_{2} \in \mathbb{K}$. Now we can normalize in the following steps:\\
 First we define $b_{0}=g_{1}^{2}$ then we replace $\mu_{2}$ by $f_{1} \cdot g_{1}$,  replace $f_{1}$ by $f_{1}- \frac{V_{1}(f_{1},f_{1})}{2V_{1}(f_{1},\mu_{2 })} \cdot \mu_{2}$ and finally we replace $\mu_{1}$ by $f_{1} \cdot \mu_{2}$.\par
 Then the classification of this case follows:\\	
\begin{equation*}
M(F)=
\left(\begin{matrix}
0  &0      & 0      & 0    \\  
0  &0      & 0     &0 \\
0  &0      & 1    &0 \\
0  &0      & 0    &1
\end{matrix}\right) , \,\, W=\langle \mu_{1},\mu_{2} \rangle \oplus \langle g_{1},f_{1} \rangle 
\end{equation*}
and $R$ is isomorphic to\\
$\mathbb{K}[\mu_{1},\mu_{2},g_{1},f_{1}]/(\mu_{1} \cdot W,\mu_{2} \cdot \mu_{2},\mu_{2} \cdot g_{1},f_{1} \cdot \mu_{2}-\mu_{1},f_{1} \cdot g_{1}-\mu_{2},f_{1}^{2}-g_{1}^{2}-\lambda \cdot \mu_{2})$\par
Moreover for the coefficient $\lambda \in \mathbb{K}$ we have the following  uniqueness result which is easy to check.
\begin{proposition}
Two actions of Case (ii.1) with coefficients $\lambda$ and $\lambda'$ respectively are equivalent if and only if $\lambda=\pm \lambda'$.
\end{proposition}
$(ii.2)$ $dim(V^{(1)} \cdot W)=2$. Then choosing $b_{0}=g_{1}^{2}$ we have $V^{(1)} \cdot W \subseteq \langle \mu_{1},b_{0} \rangle$ and $b_{0} \cdot W=0$. 
Moreover we see $f_{1} \cdot \mu_{2}=c\cdot \mu_{1}$ for some nonzero $c \in \mathbb{K}$. And we set $f_{1}^{2}=b_{0}+V_{1}(f_{1},f_{1}) \cdot \mu_{1}+d_{1} \cdot \mu_{2}$.
Now we can normalize through the follwoing steps:
\begin{align*}
g_{1} &\rightarrow g_{1}-c^{-1} \cdot V_{1}(g_{1},f_{1}) \cdot \mu_{2} &&\text{to make $g_{1} \cdot f_{1}=0$}\\
f_{1} &\rightarrow f_{1}-\frac{V_{1}(f_{1},f_{1})}{2c} \cdot \mu_{2} &&\text{to make $f_{1}^{2}=b_{0}+d_{1} \cdot \mu_{2}$}
\end{align*} 
and if $d_{1} \not =0$ (i.e., $dim(\mathfrak{m}^{2}=3)$) we replace $\mu_{2}$ by $d_{1} \cdot \mu_{2}$ to make $f_{1}^{2}=b_{0}+\mu_{2}$ then replace $\mu_{1}$ by $f_{1} \cdot \mu_{2}$ to keep $f_{1} \cdot \mu_{2}=\mu_{1}$.
This enables us to give the classification of this case:
\begin{equation*}
M(F)=
\left(\begin{matrix}
0  &0      & 0      & 0    \\  
0  &0      & 0     &0 \\
0  &0      & 1    &0 \\
0  &0      & 0    &1
\end{matrix}\right) , \,\, W=\langle \mu_{1},\mu_{2} \rangle \oplus \langle g_{1},f_{1} \rangle 
\end{equation*}
and $R$ is isomorphic to\\
$\mathbb{K}[\mu_{1},\mu_{2},g_{1},f_{1}]/(\mu_{1} \cdot W,\mu_{2} \cdot \mu_{2},\mu_{2} \cdot g_{1},f_{1} \cdot \mu_{2}-\mu_{1},f_{1} \cdot g_{1},f_{1}^{2}-g_{1}^{2}- \delta \cdot \mu_{2})$\\
where $\delta$ is the same as we define in Section 4.1.2.\\
\subcase $Ker(F) \cdot Ker(F)=0$ and $V_{(1)}=V^{(1)}$. \par
Then we can choose  $f_{1}\in W \backslash V^{(1)},g_{1} \in V^{(1)} \backslash Ker(F) $ s.t. $F(f_{1},g_{1})=1$ and $F(f_{1},f_{1})=F(g_{1},g_{1})=0$. And we divide it into two more subcases.\\\\
($iii.1$) $dim(\mathfrak{m}^{2})=2$. We set $b_{0}=g_{1} \cdot f_{1}$ and replace $\mu_{1}$ by $f_{1} \cdot \mu_{2}$, then we replace $f_{1}$ by $f_{1}-\frac{\mu_{2}}{2 V_{1}(f_{1},\mu_{2})}$ to make $f_{1}^{2}=0$. Thus we have:
\begin{align*}
g_{1} \cdot f_{1}=b_{0},f_{1} \cdot \mu_{2}=\mu_{1},f_{1}^{2}=0, g_{1}^{2}=h \cdot \mu_{1}
\end{align*}
for some $h\in \mathbb{K}$. Now if $h \not=0$ then we can furtherly make $g_{1}^{2}=\mu_{1}$ through replacing elements as the following:
\begin{align*}
\mu_{1}=\sqrt{h} \cdot \mu_{1},\,\mu_{2}=h^{\frac{1}{4}} \cdot \mu_{2},\,
f_{1}=h^{\frac{1}{4}} \cdot f_{1},\,g_{1}=h^{-\frac{1}{4}} \cdot g_{1},b_{0}=b_{0}.
\end{align*}
Then our classification of this case follows:
\begin{equation*}
M(F)=
\left(\begin{matrix}
0  &0      & 0      & 0    \\  
0  &0      & 0     &0 \\
0  &0      & 0     &1 \\
0  &0      & 1     &0
\end{matrix}\right) , \,\, W=\langle \mu_{1},\mu_{2} \rangle \oplus \langle g_{1},f_{1} \rangle 
\end{equation*}
and $R$ is isomorphic to\\
$\mathbb{K}[\mu_{1},\mu_{2},g_{1},f_{1}]/(\mu_{1} \cdot W,\mu_{2} \cdot \mu_{2},\mu_{2} \cdot g_{1},f_{1} \cdot \mu_{2}-\mu_{1},g_{1}^{2}-\mu_{1},f_{1}^{2})$\\
or\\
$\mathbb{K}[\mu_{1},\mu_{2},g_{1},f_{1}]/(\mu_{1} \cdot W,\mu_{2} \cdot \mu_{2},\mu_{2} \cdot g_{1},f_{1} \cdot \mu_{2}-\mu_{1},g_{1}^{2},f_{1}^{2})$, depending on whether $V^{(1)} \cdot V^{(1)}$ equals to zero or not.
\\\\ 
$(iii.2)$ $dim(\mathfrak{m}^{2})=3 $. We divide it into two more subcases.\\
If $dim(V^{(1)} \cdot W)=2$ then we have:
\begin{equation*}
M(F)=
\left(\begin{matrix}
0  &0      & 0      & 0    \\  
0  &0      & 0     &0 \\
0  &0      & 0    &-1\\
0  &0      & -1    &0
\end{matrix}\right) , \,\, W=\langle \mu_{1},\mu_{2} \rangle \oplus \langle g_{1},f_{1} \rangle 
\end{equation*}
and $R$ is isomorphic to\\
$\mathbb{K}[\mu_{1},\mu_{2},g_{1},f_{1}]/(\mu_{1} \cdot W,\mu_{2} \cdot \mu_{2},\mu_{2} \cdot g_{1},f_{1} \cdot \mu_{2}-\mu_{1},g_{1}^{2}-\mu_{1},f_{1}^{2}-\mu_{2})$\\
If $dim(V^{(1)} \cdot W)=3$ then we have:
\begin{equation*}
M(F)=
\left(\begin{matrix}
0  &0      & 0      & 0    \\  
0  &0      & 0     &0 \\
0  &0      & 0    &1 \\
0  &0      & 1    &0
\end{matrix}\right) , \,\, W=\langle \mu_{1},\mu_{2} \rangle \oplus \langle g_{1},f_{1} \rangle 
\end{equation*}
and $R$ is isomorphic to\\
$\mathbb{K}[\mu_{1},\mu_{2},g_{1},f_{1}]/(\mu_{1} \cdot W,\mu_{2} \cdot \mu_{2},\mu_{2} \cdot g_{1},f_{1} \cdot \mu_{2}-\mu_{1},g_{1}^{2}-\mu_{2},f_{1}^{2}-\mu_{2})$\\
or \\
$\mathbb{K}[\mu_{1},\mu_{2},g_{1},f_{1}]/(\mu_{1} \cdot W,\mu_{2} \cdot \mu_{2},\mu_{2} \cdot g_{1},f_{1} \cdot \mu_{2}-\mu_{1},g_{1}^{2}-\mu_{2},f_{1}^{2})$\\
where one can easily check these two actions are not equivalent.
\case $dim(W)=3$ Then we have two subcases as follows.
\subcase $Ker(F) \cdot Ker(F) \not =0$ then we have:
\begin{equation*}
M(F)=
\left(\begin{matrix}
0      & 0     &0 \\
0      & 0    &0 \\
0      & 0    &1
\end{matrix}\right) , \,\, W=\langle \mu_{1},\mu_{2} \rangle \oplus \langle e \rangle 
\end{equation*}
and $R$ is isomorphic to\\
$\mathbb{K}[\mu_{1},\mu_{2},e]/(\mu_{1} \cdot W,\mu_{2} \cdot \mu_{2}-\mu_{1},\mu_{2} \cdot e,e^{3}-\mu_{1})$, if  $dim(\mathfrak{m}^{2})=2$\ and $\mathfrak{m}^{3} \not =0$.\\
$\mathbb{K}[\mu_{1},\mu_{2},e]/(\mu_{1} \cdot W,\mu_{2} \cdot \mu_{2}-\mu_{1},\mu_{2} \cdot e,e^{3})$, if  $dim(\mathfrak{m}^{2})=2$\ and $\mathfrak{m}^{3} =0$.\\
$\mathbb{K}[\mu_{1},\mu_{2},e]/(\mu_{1} \cdot W,\mu_{2} \cdot \mu_{2}-\mu_{1},\mu_{2} \cdot e,e^3-\mu_{2})$, if $dim(\mathfrak{m}^{2})=3$.\\
\subcase $Ker(F) \cdot Ker(F)=0$ then we have:\\
\begin{equation*}
M(F)=
\left(\begin{matrix}
0      & 0     &0 \\
0      & 0    &0 \\
0      & 0    &1
\end{matrix}\right) , \,\, W=\langle \mu_{1},\mu_{2} \rangle \oplus \langle e \rangle 
\end{equation*}
and $R$ is isomorphic to\\
$\mathbb{K}[\mu_{1},\mu_{2},e]/(\mu_{1} \cdot W,\mu_{2} \cdot \mu_{2},\mu_{2} \cdot e-\mu_{1},e^{3}-\mu_{1})$, if  $dim(\mathfrak{m}^{2})=2$ and $\mathfrak{m}^{3} \not=0$.\\
$\mathbb{K}[\mu_{1},\mu_{2},e]/(\mu_{1} \cdot W,\mu_{2} \cdot \mu_{2},\mu_{2} \cdot e-\mu_{1},e^{3})$, if  $dim(\mathfrak{m}^{2})=2$ and $\mathfrak{m}^{3}=0$.\\
$\mathbb{K}[\mu_{1},\mu_{2},e]/(\mu_{1} \cdot W,\mu_{2} \cdot \mu_{2},\mu_{2} \cdot e-\mu_{1},e^3-\mu_{2})$, if $dim(\mathfrak{m}^{2})=3$
\end{mycases}

\section*{Acknowledgments}
 I am grateful to Baohua Fu for introducing the problem, guidance and revising this paper. I am also grateful to Zhijun Luo for helpful discussions. I would like to thank the referee for many valuable comments and suggestions.

\end{document}